\documentclass[11pt,letterpaper]{amsart}
\usepackage[lmargin=1.25in,rmargin=1.25in,tmargin=1in,bmargin=1in]{geometry}

\pdfoutput=1

\usepackage{
	amsmath,			
	amssymb,			
    amsthm,
    bm,
	enumerate,		
    float,
	graphicx,			
    arydshln,
    thmtools,
	lastpage,			
	multicol,			
	multirow,			
    parskip,
    pdfpages,
    setspace,
    tikz,
    tikzscale,
    thm-restate,
    verbatim,
	xcolor
}

\usetikzlibrary{
	arrows,
    arrows.meta,
    calc,
    decorations.markings,
	decorations.pathmorphing,
	matrix,
	positioning,
    shapes.arrows,
	shapes.geometric
}
\usepackage{pgfplots}
\pgfplotsset{compat=newest}

\usepackage[pdfencoding=auto, psdextra]{hyperref}
\usepackage{cleveref}

\hypersetup{
    colorlinks=true,
    linkcolor=blue,
    filecolor=magenta, 
    citecolor=cyan,
}

\usepackage[T1]{fontenc}
\usepackage{charter}

\newcounter{dfn}

\newcounter{example}

\newcounter{lem}

\newcounter{thm}

\newtheorem{lemma}{Lemma}

\newtheorem{prop}{Proposition}

\theoremstyle{definition}

\theoremstyle{remark}
\newtheorem*{remark}{Remark}

\theoremstyle{remark}
\newtheorem{claim}{Claim}

\theoremstyle{remark}

\makeatletter
\newcommand{\addresseshere}{%
  \enddoc@text\let\enddoc@text\relax
}
\makeatother

\usepackage{fancyhdr}

\newcommand{\RR}{\mathbb{R}}

\newcommand{\ZZ}{\mathbb{Z}}
\newcommand{\K}{K_{3,3}}

\begin{document}

\title{Toroidal embeddings of cubic projective plane obstructions}
\author{Marie Kramer}
\address{Department of Mathematics, Syracuse University, Syracuse, NY USA}
\email{mkrame04@syr.edu}
\begin{abstract}
    Work of Glover and Huneke shows that a cubic graph embeds into the real projective plane if and only if it does not contain one of six topological minors called cubic projective plane obstructions. Here we classify up to equivalence the embeddings of these graphs into the torus.
\end{abstract}

\begingroup
\let\clearpage\relax
\maketitle
\section*{Introduction}
\label{sec:intro}

Kuratowski's Theorem is a classical result: A graph is planar if and only if it does not contain $K_5$ or $\K$ as a topological minor, that is, a subgraph isomorphic to a subdivision $K_5$ or $\K$ \cite{Kuratowski30}. We say $K_5$ and $\K$ are the topological obstructions for the plane.

Via stereographic projection, embeddability into the plane is equivalent to embeddability into the $2$-sphere. Hence, $\{K_5, \K\}$ is the complete set of topological obstructions for the orientable surface of genus zero. This raises the question of embedding these graphs into other surfaces. Gagarin, Kocay, and Neilson have shown that up to equivalence there are six embeddings of $K_5$ into the torus and two embeddings of $\K$ \cite{GagarinKocayNeilson03}. Similarly, Mohar proved that up to equivalence there are two embeddings of $K_5$ into the projective plane and one embedding of $\K$ \cite{Mohar93}.

In general, it is easy to see that obstructions for the orientable surface of genus $n$ embed into the orientable surface of genus $n+1$. Besides some elementary observations (see for example \textsection 4.4 in \cite{GraphsSurfaces}), there seems to be less known about the relation between embeddability into orientable surfaces and nonorientable surfaces. One result in this area due to Fiedler, Huneke, Richter, and Robertson shows how to compute the orientable genus for graphs embeddable into the projective plane \cite{FiedlerHunekeRichterRobertson95}.

Work by Archdeacon, Glover, Huneke, and Wang gives an analogue to Kuratowski's Theorem for the projective plane: There are 103 topological obstructions for the projective plane, six of which are cubic \cite{GloverHuneke75}, \cite{GloverHunekeWang79}, \cite{Archdeacon81}.

\begin{figure}[H]
    \centering
    \includegraphics[height=1.7cm]{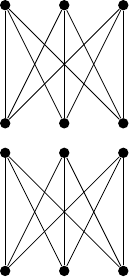} \hspace{.1cm}
    \includegraphics[height=1.6cm]{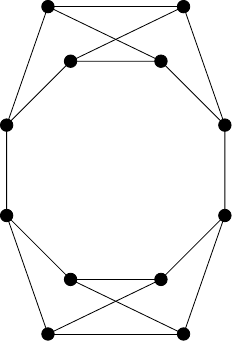} \hspace{.1cm}
    \includegraphics[height=1.6cm]{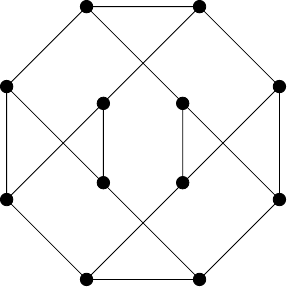} \hspace{.1cm}
    \includegraphics[height=1.6cm]{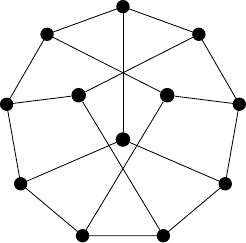} \hspace{.1cm}
    \includegraphics[height=1.6cm]{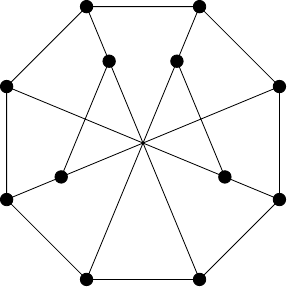} \hspace{.1cm}
    \includegraphics[height=1.6cm]{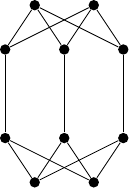} 
    \caption{The six cubic projective plane obstructions}
    \label{fig:6CubicRP2}
\end{figure}

In this paper, we investigate the embeddability of the six cubic projective plane obstructions into the torus. We will refer to these by $E_{42}, \ F_{11}, \ F_{12}, \ F_{13}, \ F_{14}, \ G_{1}$ following the notation used in \cite{GloverHunekeWang79} and \cite{GraphsSurfaces}.

\begin{restatable}{thmx}{DistinctEmbeddings}
    \label{thm:ProjectiveObstructionsInTorus}
    The graph $E_{42}$ does not embed into the torus, the graph $F_{12}$ embeds into the torus in exactly four ways up to equivalence, and each of the remaining four cubic projective plane obstructions embeds into the torus in exactly two ways up to equivalence.
\end{restatable}

We refer to the Appendix for illustrations of the embeddings.
Theorem \ref{thm:ProjectiveObstructionsInTorus} is proved in Lemma \ref{lem:E42Torus} for the case of $E_{42}$ and in Sections \ref{sec:F11Lemma} -- \ref{sec:G1Lemma} for the other five graphs.

Efforts have been made to obtain a complete list of obstructions for the torus, but no such list has been found yet. However, there are a number of partial results. There are currently over 250,000 known topological torus obstructions \cite{MyrvoldWoodcock18}. Chambers has investigated cubic torus obstructions with the aid of computers \cite{Chambers02}, and it is believed that his list of 206 cubic obstructions is complete \cite{MyrvoldWoodcock18}. 
One application of Theorem \ref{thm:ProjectiveObstructionsInTorus} is to provide an independent proof that this list is indeed complete for small Betti numbers. This will be the subject of future work.

Further motivation for this work comes from an application to Riemannian geometry discovered by Kennard, Wiemeler and Wilking in \cite{KennardWiemelerWilking22}. More specifically, if a graph embeds into a surface, one obtains bounds on its systole. These systole bounds give bounds on matroid theoretic invariants, which in turn give geometric bounds for special tours representations. It turns out that cubic graphs represent extreme cases in the context of systole bounds and are therefore of high interest. 
While we restrict our attention to cubic graphs in this paper, we imagine a similar strategy would work to classify the embeddings into the torus of the remaining topological projective plane obstructions.

\textbf{Acknowledgements.} The author is grateful for support from NSF Research Grants DMS-2005280 and DMS-2402129, as well as from the Syracuse University Graduate School through Pre-Dissertation and Dissertation Fellowships in Summer 2023 and Summer 2024, respectively. The author would like to thank her advisor Dr. Lee Kennard for the invaluable advice, continued support, and stimulating conversations.
\endgroup

\section{Preliminaries}
\label{sec:prelim}

Unless otherwise specified $G=(V,E)$ denotes an undirected, simple graph.
Most of the graphs we work with are \textbf{\textit{cubic}}, that is, each of their vertices has degree three.

An \textbf{\textit{embedding}} of a graph $G$ into a surface $\Sigma$ is a function $\varphi$ mapping the vertices of $G$ to points in $\Sigma$, and the edges of $G$ to continuous curves in $\Sigma$, such that curves representing distinct edges intersect only at images of vertices (see \textsection 15.1 in \cite{GraphsAlgosAndOpti}). We denote an embedding by $\varphi: G \hookrightarrow \Sigma$.
A \textbf{face} of an embedding is a connected component of $\Sigma \setminus \varphi(G)$.
A \textbf{\textit{facial cycle}} or \textbf{\textit{facial walk}} is an oriented cycle or walk in $G$ along the boundary of a face of $\varphi(G)$.

A graph $G$ is a \textbf{\textit{(topological) obstruction}} for a surface $\Sigma$ if $G$ does not embed into $\Sigma$ but every (topological) minor does. Here, a \textbf{\textit{topological minor}} is a graph obtained from $G$ by deleting edges and/or vertices and contracting edges with an endpoint of degree two. In other words, a graph is a topological minor of $G$ if it has a  subdivision isomorphic to a subgraph of $G$. A graph is a \textbf{\textit{cubic obstruction}} for a surface $\Sigma$ if it is cubic and a topological obstruction for $\Sigma$.
We observe that if $G$ and $H$ are cubic graphs, $H$ is a topological minor of $G$ if and only if $H$ is a minor of $G$. Hence, restricted to the class of cubic graphs, topological obstructions and obstructions are equivalent.
 
Two embeddings $\varphi_1$ and $\varphi_2$ of a graph $G$ into a surface $\Sigma$ are \textbf{\textit{equivalent}} if there exists a homeomorphism $h: \Sigma \to \Sigma$ such that $\varphi_1(G) = (h \circ \varphi_2)(G)$ (see \textsection 15.2 in \cite{GraphsAlgosAndOpti}). That is, there exists a homeomorphism mapping the image of $G$ under $\varphi_2$ to the image of $G$ under $\varphi_1$.
Since any homeomorphism of a surface maps facial walks to facial walks, two embeddings with facial walks of different lengths cannot be equivalent.

\section{Embeddings of \texorpdfstring{$K_{3,3}$}{K\_{3,3}} into the torus}
\label{sec:embedK33}

It is well known that $K_{3,3}$ is the only cubic obstruction for the plane, as shown by Kuratowski \cite{Kuratowski30}. It is also known that the embedding of $K_{3,3}$ into the projective plane is unique up to equivalence \cite{Mohar93}. Similarly, $K_{3,3}$ embeds into the torus, and there are exactly two such embeddings up to equivalence \cite{GagarinKocayNeilson03}. Observe that the embedding on the left in Figure \ref{fig:K33InTorus} has two facial $4$-cycles and one facial $10$-cycle, while the embedding on the right has three facial $6$-cycles. 
    \begin{figure}[H]
        \centering
        \includegraphics[height=2cm]{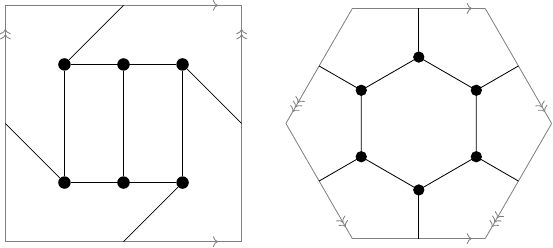}
        \caption{$K_{3,3}$ in the torus}
        \label{fig:K33InTorus}
    \end{figure}

Notice that each face of the toroidal embeddings of $K_{3,3}$ is homeomorphic to the plane, so the proof of Theorem \ref{thm:ProjectiveObstructionsInTorus} for the graph $E_{42}$ follows immediately.

\begin{lemma}
    \label{lem:E42Torus}
    There are no embeddings of $E_{42}$ into the torus.
\end{lemma}

\begin{proof}
    Recall that $E_{42}=K_{3,3} \sqcup K_{3,3}$. As seen in \cite{GagarinKocayNeilson03}, there are exactly two inequivalent embeddings of $K_{3,3}$ into the torus, pictured in Figure \ref{fig:K33InTorus}. Therefore, one copy of $K_{3,3}$ is embedded into the torus in one of those two ways and the second copy has to be embedded into one of the faces.
    In both cases, each face of the embedding is homeomorphic to a disk. As $K_{3,3}$ is nonplanar by \cite{Kuratowski30}, it follows that $E_{42}$ does not embed into the torus.
\end{proof}

\begin{remark}
    Alternatively, it follows from \cite{BattleHararyKodama62} that $E_{42}$ is one of the disconnected torus obstructions. Note that this implies that $E_{42}$ embeds into the double torus, the orientable surface of genus two.
\end{remark}

In order to prove Theorem \ref{thm:ProjectiveObstructionsInTorus} for the other five cubic projective plane obstructions, we use a classification of (directed) edged embeddings of $\K$ into the torus. By an \textit{edged embedding}, we mean an embedding of $\K$ with one highlighted edge. In a \textit{directed edged embedding}, the highlighted edge has an orientation. We will first prove a statement about the number of directed edged embeddings. The analogous statement about the number of edged embeddings will then follow immediately.

\begin{lemma}
        \label{lem:ClassDirEdgeK33}
        There are exactly six inequivalent directed edged embeddings of $\K$ into the torus.
        \begin{figure}[H]
            \centering
            \includegraphics[height=2cm]{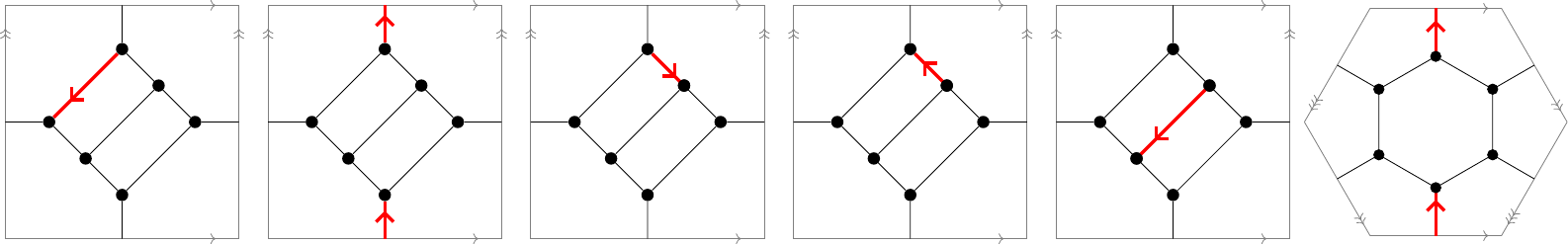}
            \caption{The six directed edged embeddings $(\K, \vec e) \hookrightarrow T^2$}
            \label{fig:DirEdgedK33InT}
        \end{figure}
    \end{lemma}

    \begin{proof}
        We will start by showing that there are at most six equivalence classes of directed edged embeddings of $\K$ into the torus.
        Note that $\K$ has nine edges, so there are eighteen ways to select and orient an edge. However, we will see that many of these choices result in equivalent embeddings.

        First, consider the embedding of $\K$ into the torus pictured on the left in Figure \ref{fig:K33InTorus}. We will use the following two homeomorphisms of the torus to prove the below claims: A rotation $\rho$ by $180^\circ$, and a reflection $r$ along the diagonal from the bottom left to the top right. Note that these operations map the image of $\K$ to the image of $\K$. Embeddings related by these operations are thus equivalent according to the definition in Section \ref{sec:prelim}.
        
        \begin{claim}
            \label{claim:DirEdg1}
            The four directed edged embeddings $(\K, \vec e) \hookrightarrow T^2$ shown below are equivalent.
            \begin{figure}[H]
                \centering
                \includegraphics[height=2cm]{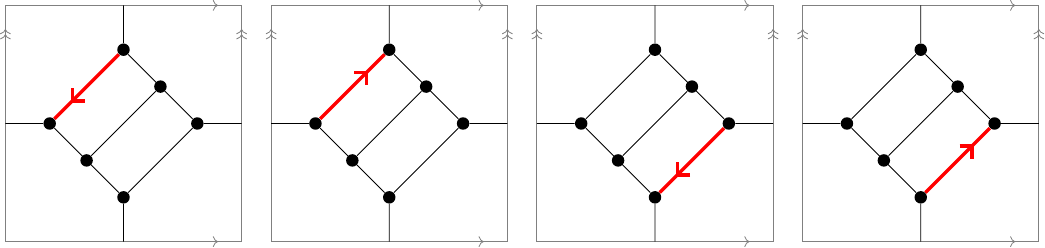}
            \end{figure}
            \textit{Proof. \hspace{.05cm}}
            The first and fourth embedding, and the second and third one are related by a rotation by $180^\circ$, respectively. The first and third embedding, and the second and fourth one are related by the above mentioned reflection $r$, respectively.
            $\blacksquare$
        \end{claim}

        Arguing similarly, we easily get Claims \ref{claim:DirEdg2} -- \ref{claim:DirEdg5} below.

        \begin{claim}
            \label{claim:DirEdg2}
            The four directed edged embeddings $(\K, \vec e) \hookrightarrow T^2$ shown below are equivalent.
            \begin{figure}[H]
                \centering
                \includegraphics[height=2cm]{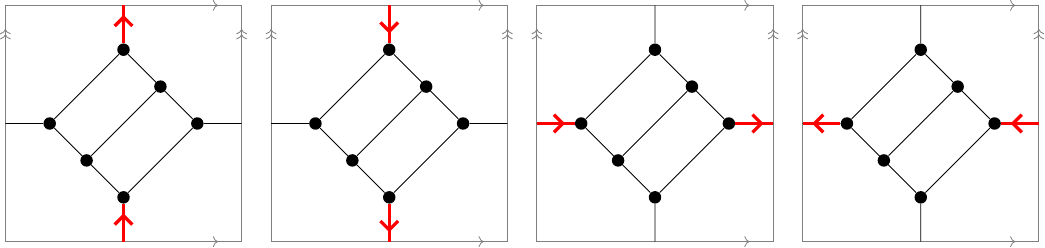}
            \end{figure}
        \end{claim}

        \begin{claim}
            \label{claim:DirEdg3}
            The four directed edged embeddings $(\K, \vec e) \hookrightarrow T^2$ shown below are equivalent.
            \begin{figure}[H]
                \centering
                \includegraphics[height=2cm]{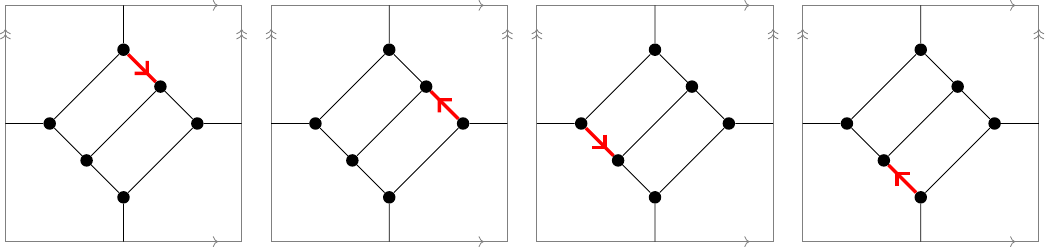}
            \end{figure}
        \end{claim}

        \begin{claim}
            \label{claim:DirEdg4}
            The four directed edged embeddings $(\K, \vec e) \hookrightarrow T^2$ shown below are equivalent.
            \begin{figure}[H]
                \centering
                \includegraphics[height=2cm]{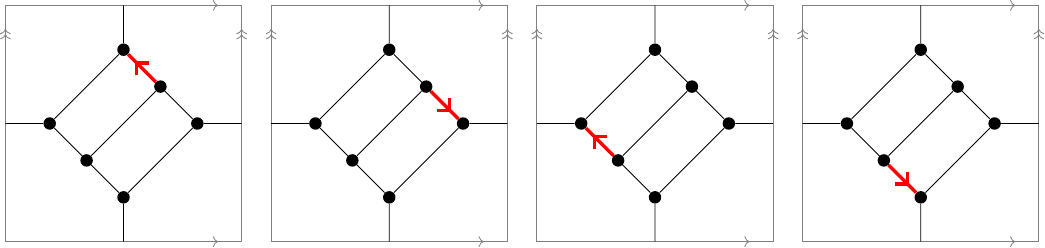}
            \end{figure}
        \end{claim}

        \begin{claim}
            \label{claim:DirEdg5}
            The two directed edged embeddings $(\K, \vec e) \hookrightarrow T^2$ shown below are equivalent.
            \begin{figure}[H]
                \centering
                \includegraphics[height=2cm]{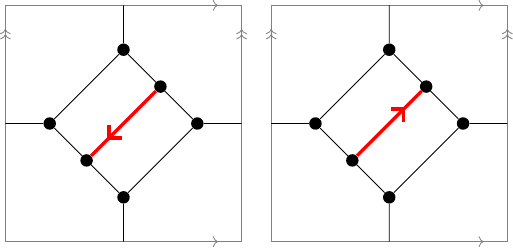}
            \end{figure}
        \end{claim}

        \begin{remark}
            Note that the automorphism group of $\K$ is $(S_3 \times S_3) \rtimes \ZZ_2$, where $\ZZ_2$ acts by interchanging the $S_3$ factors.
            The homeomorphisms $\rho$ and $r$ of the embedding of $\K$ in the left of Figure \ref{fig:K33InTorus} induce a subgroup of order four of the automorphism group of $\K$.
            More specifically, the group generated by $\rho$ and $r$ is of the from $\langle \rho, r \ \vert \ \rho^2=r^2=e, \rho r \rho^{-1} = r\rangle \cong \ZZ_2 \times \ZZ_2$.
            Lastly, Claims \ref{claim:DirEdg1} -- \ref{claim:DirEdg5} give that the action of $\ZZ_2 \times \ZZ_2$ on the 18-element set of oriented edged embeddings of $\K$ into the torus has four orbits of size four, and one orbit of size two.
        \end{remark}

         Next, we consider the embedding of $\K$ into the torus pictured on the right in Figure \ref{fig:K33InTorus}. We will use the following two homeomorphisms of the torus to prove the claim below: A counterclockwise rotation by $60^\circ$, and the operation $\varepsilon$ obtained by cutting the torus into six triangular regions and pasting them back together as pictured below. Note that both of these operations map the the image of $\K$ to the image of $\K$.
        \begin{center}
            \includegraphics[height=2.5cm]{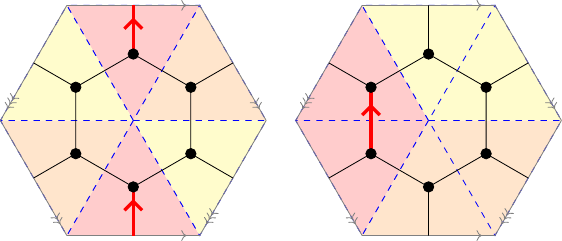}
        \end{center} 
        
        \begin{claim}
            \label{claim:DirEdg6}
            All 18 directed edged embedding $(\K, \vec e) \hookrightarrow T^2$ with underlying embeddings as on the right of Figure \ref{fig:K33InTorus} are equivalent.\\
            \textit{Proof. \hspace{.05cm}} The homeomorphism $\varepsilon$ shows the equivalence of a directed edged embedding with $\vec e$ being one of the six edges of the central hexagon to one with $\vec e$ being one of the other three edges. Moreover, in composition with counterclockwise rotations by $60^\circ$, $\varepsilon$ can be used to reverse the direction of a directed edge in the central hexagon. Finally, repeated application of counterclockwise rotations by $60^\circ$ gives the remaining desired equivalences.
            $\blacksquare$ 
        \end{claim}

        \begin{remark}
            These homeomorphisms induce automorphisms of $\K$. The subgroup of automorphisms that can be realised by this embedding of $\K$ into the torus is isomorphic to $S_3 \times S_3$, however it is not the the standard such subgroup $(S_3 \times S_3) \times 1 \leqslant (S_3 \times S_3) \rtimes \ZZ_2$.
        \end{remark}

         Combining Claims \ref{claim:DirEdg1} -- \ref{claim:DirEdg6} gives that there are at most six inequivalent embeddings. It remains to show that the embeddings pictured in Figure \ref{fig:DirEdgedK33InT} are pairwise inequivalent.
         Since Embedding 6 is the only one with facial $6$-cycles and homeomorphisms of the torus map facial cycles to facial cycles, Embedding 6 cannot be equivalent to any of the other embeddings.
         Observe that in Embeddings 1 through 5 there are exactly two edges that appear in the facial $10$-cycle twice. Such an edge has to get mapped to such an edge under homeomorphisms of the torus. In particular, Embedding 2 cannot be equivalent to any of the other four remaining ones.
         Similarly, in Embeddings 1 through 5 there exists a unique edge separating the two facial $4$-cycles. Any homeomorphism of the torus has to map this edge to itself. This implies in particular that Embedding 5 cannot be equivalent to any of the other three remaining embeddings.
         It remains to compare Embeddings 1, 3, and 4. In Embedding 1, neither the starting point nor the endpoint $\vec e$ are part of both facial $4$-cycle. In Embedding 3, the endpoint of $\vec e$ is part of both facial $4$-cycles, and in Embedding 4 the starting point of $\vec e$ is part of both facial $4$-cycles. This implies that Embeddings 1, 3, and 4 are pairwise inequivalent, finishing the proof.
    \end{proof}

\begin{lemma}
    \label{lem:ClassEdgedK33}
    There are exactly five inequivalent edged embeddings of $\K$ into the torus.
    \begin{figure}[H]
            \centering
            \includegraphics[height=2cm]{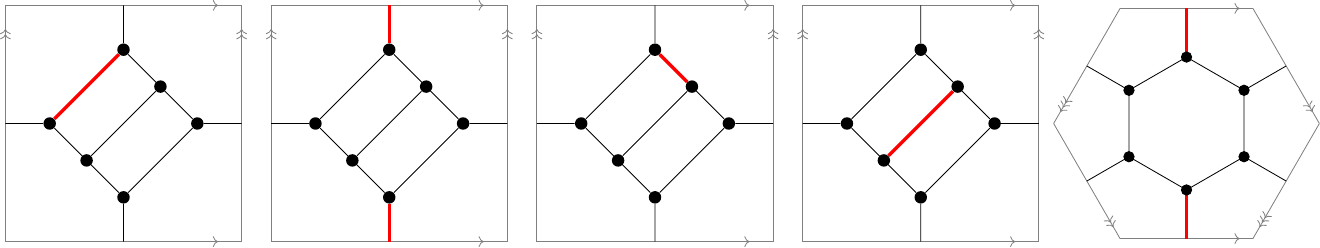}
            \caption{The five edged embeddings $(\K, e^\ast) \hookrightarrow T^2$}
            \label{fig:EdgedK33InT}
        \end{figure}
\end{lemma}

\begin{proof}
    This follows directly from Lemma \ref{lem:ClassDirEdgeK33} by disregarding the orientation. In particular, Embeddings 3 and 4 in Figure \ref{fig:DirEdgedK33InT} are the same if we delete the orientation. 
\end{proof}

 In addition to (directed) edged embeddings, we will use the notion of a \textit{cycled embedding}. By this, we mean an embedding of $\K$ into the torus with a highlighted $4$-cycle.

\begin{lemma}
    \label{lem:ClassCycleK33}
    The are exactly five inequivalent cycled embeddings of $\K$ into the torus.
    \begin{figure}[H]
            \centering
            \includegraphics[height=2cm]{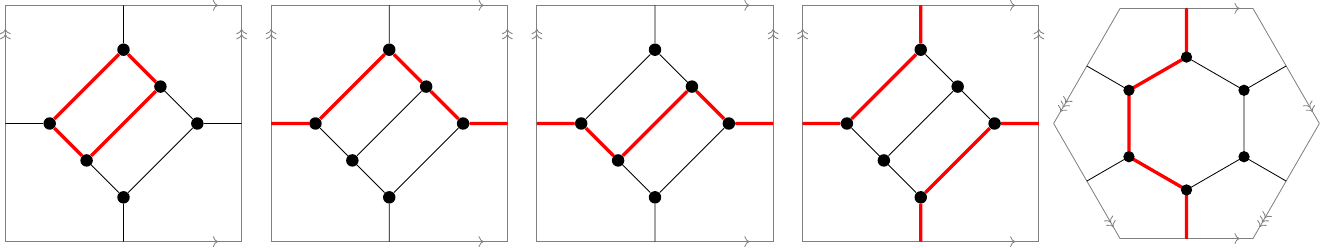}
            \caption{The five cycled embeddings $(\K, C_4) \hookrightarrow T^2$}
            \label{fig:CycledK33InT}
        \end{figure}
\end{lemma}

\begin{proof}
    Similar to the proof of Lemma \ref{lem:ClassDirEdgeK33}, we will start by showing that there are at most five equivalence classes of cycled embeddings of $\K$ into the torus.
    Note that there are nine distinct $4$-cycles in $K_{3,3}$: Let $V(\K)=V_1 \cup V_2$ be the partition of vertices of $\K$. Any $4$-cycle contains exactly two vertices from $V_1$ and two vertices from $V_2$. Thus, the total number of choices for the set of vertices of a $4$-cycle is $\binom{3}{2} \cdot \binom{3}{2} = 9$. Once the vertices of $C_4 \subset \K$ are fixed, the edges of $C_4$ are completely determined.

     First, we consider the embedding of $\K$ into the torus pictured on the left of Figure \ref{fig:K33InTorus}. Recall the homeomorphisms $\rho$, the rotation by $180^\circ$, and $r$, the reflection along the diagonal from the bottom left to the top right, used in the proof of Lemma \ref{lem:ClassDirEdgeK33}. The next three claims follow as in the proof of Lemma \ref{lem:ClassDirEdgeK33}, so we omit the proofs.

    \begin{claim}
        \label{claim:Cycled1}
        The two cycled embeddings $(\K, C_4) \hookrightarrow T^2$ shown below are equivalent.
        \begin{center}
            \includegraphics[height=2cm]{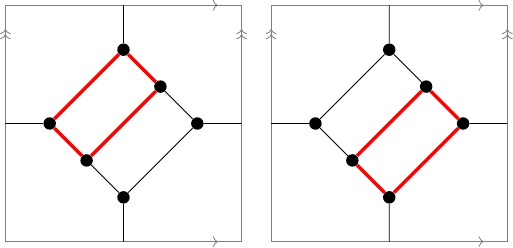}
        \end{center}
    \end{claim}

    \begin{claim}
        \label{claim:Cycled2}
        The four cycled embeddings $(\K, C_4) \hookrightarrow T^2$ shown below are equivalent.
        \begin{center}
            \includegraphics[height=2cm]{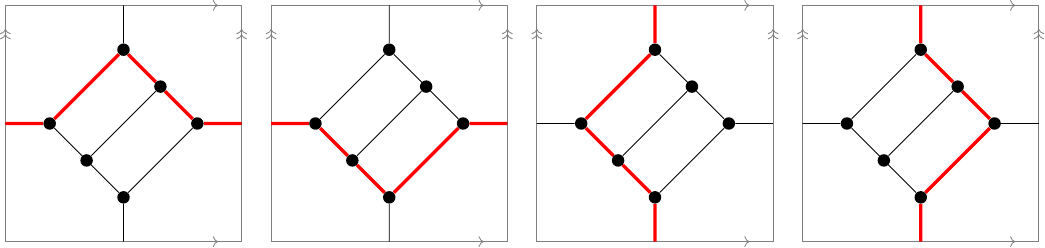}
        \end{center}
    \end{claim}

    \begin{claim}
        \label{claim:Cycled3}
        The two cycled embeddings $(\K, C_4) \hookrightarrow T^2$ shown below are equivalent.
        \begin{center}
            \includegraphics[height=2cm]{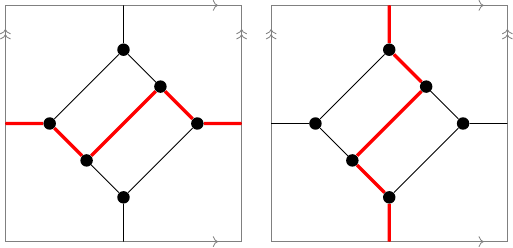}
        \end{center}
    \end{claim}

    \begin{remark}
        The action of $\ZZ_2 \times \ZZ_2$ on the 9-element set of cycled embeddings of $\K$ into the torus has two orbits of size two, one orbit of size four, and one orbit of size one.
    \end{remark}

     Next, we consider the embedding of $\K$ into the torus pictured on the right in Figure \ref{fig:K33InTorus}. As in the proof of Lemma \ref{lem:ClassDirEdgeK33}, we will use the counterclockwise rotation by $60^\circ$ and the homeomorphism $\varepsilon$.

    \begin{claim}
        \label{claim:Cycled5}
        All nine cycled embeddings $(\K, C_4) \hookrightarrow T^2$ shown below are equivalent.
        \begin{center}
            \includegraphics[width=.95\textwidth]{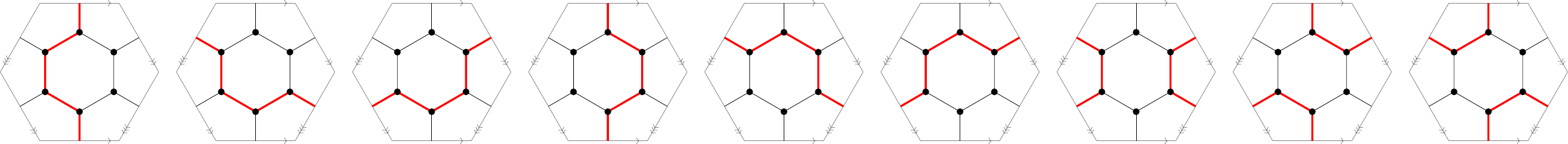}
        \end{center}
        \text{ } \textit{Proof. \hspace{.05cm}} Going from left to right, we can obtain the first six embeddings from the first one by a counterclockwise rotation by an integer multiple of $60^\circ$. 
        We can obtain the seventh embedding from the first one by an application of $\varepsilon$. 
        The final two embeddings can be obtained from the seventh one by a counterclockwise rotation by an integer multiple of $60^\circ$.
        Since equivalence of embeddings is an equivalence relation, the claim follows.
        $\blacksquare$
    \end{claim}

     Combining Claims \ref{claim:Cycled1} -- \ref{claim:Cycled5} implies that there are at most five inequivalent embeddings. It remains to show that the embeddings pictured in Figure \ref{fig:CycledK33InT} are pairwise inequivalent.
    Since Embedding 5 is the only one with facial $6$-cycles, Embedding 5 cannot be equivalent to any of the other embeddings.
    Observe that in Embeddings 1 through 4 there are exactly two edges that appear in the facial $10$-cycle twice. Such an edge has to get mapped to such an edge under homeomorphisms of the torus. In particular, as Embedding 4 is the only one whose highlighted $4$-cycle uses both of these edges, it cannot be equivalent to any of the other three remaining ones.
    Similarly, in Embeddings 1 through 4 there exists a unique edge separating the two facial $4$-cycles. Any homeomorphism of the torus has to map this edge to itself. This implies in particular that Embedding 2 cannot be equivalent to any of the other two remaining ones as the highlighted $4$-cycle in embedding 2 does not use that edge.
    It remains to compare Embeddings 1 and 3. In Embedding 3, the highlighted $4$-cycle uses an edge appearing twice in the facial $10$-cycle, whereas the $4$-cycle in Embedding 1 does not. This implies that Embeddings 1 and 3 are inequivalent, finishing the proof.
\end{proof}

\section{Toroidal embeddings of \texorpdfstring{$F_{11}$}{F\_{11}}}
\label{sec:F11Lemma}

\setcounter{claim}{0}

In the remaining sections, we use repeatedly the following consequence of the Jordan Curve Theorem. A proof of the below proposition can be found in  \textsection 2.2 in \cite{GraphsSurfaces}. We present a different proof here.

\begin{prop}
    \label{lem:JCT_consequenc}
    Let $C$ be a simple closed curve in $\RR^2$, $x, y \in C$, and $P: I \to \RR^2$ a continuous path from $x$ to $y$. Then $\RR^2 \setminus (C \cup P)$ has three components.
\end{prop}

\begin{proof}  
    By the Jordan Curve Theorem, $C=C_1 \cup C_2$ separates the plane into two components, called the interior and exterior. In particular, the interior is nonempty. If $P$ is on the interior of $C$, $P \cup C_1$ and $P \cup C_2$ are simple closed curves, and thus have a nonempty interior as well. Hence, we have the claimed three components.
        \begin{center}
            \includegraphics[height=3cm]{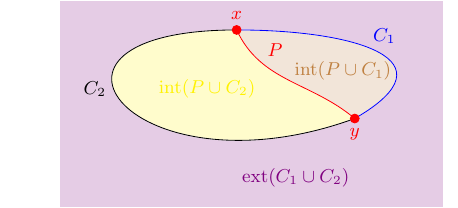}
        \end{center} 
    The case of $P$ being in the exterior of $C$ may be reduced to the just covered case by stereographic projection.
\end{proof}

\begin{lemma}
\label{lem:F11TorusEmbeddings}
    There are exactly two inequivalent unlabelled embedding of $F_{11}$ into the torus.
\end{lemma}

\begin{figure}[H]
        \centering
        \includegraphics[height=3cm]{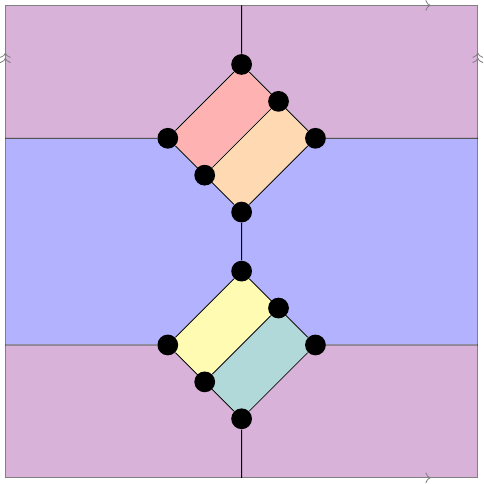} \hspace{.1cm}
        \includegraphics[height=3cm]{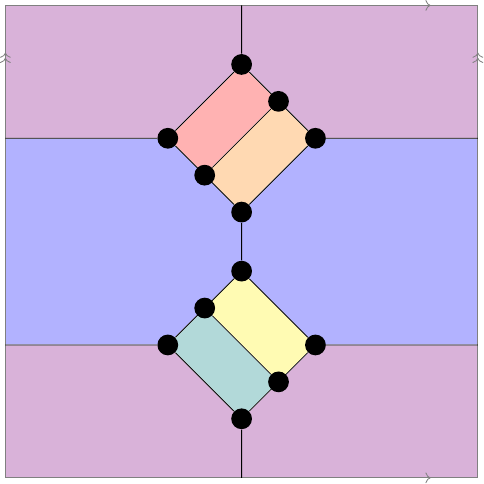}
        \caption{The two embeddings of $F_{11}$ into the torus}
        \label{fig:TwoF11Embeddings}
    \end{figure}

    \begin{remark}
        \noindent Both embeddings have four facial cycles of length four, and two facial cycles of length ten. To distinguish between the two, we observe that both embeddings have four vertices that are each part of two facial $4$-cycles. We label these vertices $a, b, c, d$ as shown in the figure below. Moreover, vertices $a$ and $d$ lie on a common facial cycle, and $b$ and $c$ lie on common facial cycle. Denote these by $C^{ad}_L, C^{ad}_R$, $C^{bc}_L$ and $C^{bc}_R$, respectively. We observe that on the one hand, vertices $a$ and $d$ have distance five in $C^{ad}_L$, and so do vertices $b$ and $c$ in $C^{bc}_L$. On the other hand, vertices $a$ and $d$ have distance three in $C^{ad}_R$, and so do vertices $b$ and $c$ in $C^{bc}_R$. Thus, the two embeddings cannot be equivalent. We will refer to the first of these two embeddings as the \textit{5-embedding}, and the other as the \textit{3-embedding}. 
        \begin{center}
            \includegraphics[width=.45\textwidth]{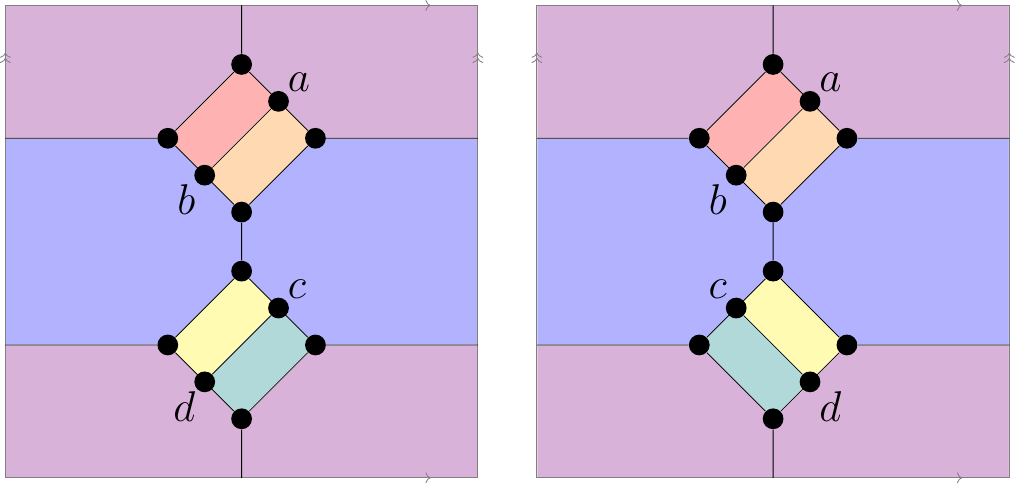}
        \end{center}
    \end{remark}

\begin{proof}
    Since $F_{11}$ is a cubic nonprojective graph, it is nonplanar and contains a subgraph homeomorphic to $K_{3,3}$. Furthermore, any subgraph homeomorphic to $K_{3,3}$ is embedded into the torus in one of the ways pictured in Figure \ref{fig:K33InTorus}. To classify the embeddings of $F_{11}$, we fix a subgraph $H$ homeomorphic to $\K$ as shown below. Note that $V(H)=V(F_{11})$ and $E(H)=E(F_{11}) \setminus \{(i, i+3)\}_{i \in \{1,2,3\}}$. Let $e^\ast$ denote the path of length seven in $H$ containing the vertices $1$ through $6$.
    \begin{figure}[H]
        \centering
        \includegraphics[height=2.5cm]{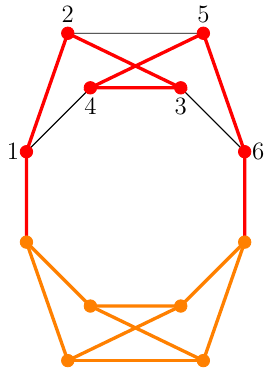}
    \end{figure}
    The graph $H$ is shown below, where $\K$ is drawn as the Möbius ladder on three rungs. Once we have the pictured subdivision of $\K$, there is a unique way to add the remaining edges to obtain $F_{11}$ by recalling that $F_{11}$ is cubic and has girth four.
    \begin{figure}[H]
        \centering
        \includegraphics[height=2cm]{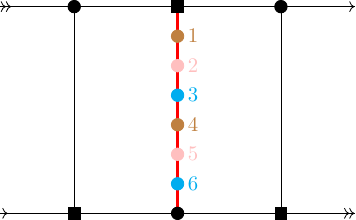} \hspace{1cm}
        \includegraphics[height=2cm]{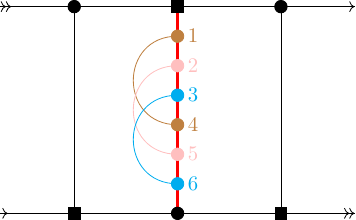}
    \end{figure}

    By Lemma \ref{lem:ClassEdgedK33}, there are five inequivalent edged embeddings $(\K, e^\ast) \hookrightarrow T^2$. As described above, picking one edge in $\K$ to be the edge $e^\ast$ completely determines how to build $F_{11}$. Hence, we analyse the five cases below and check which ones can be extended to an embedding of $F_{11}$.
    \begin{figure}[H]
        \centering
        \includegraphics[width=.7\textwidth]{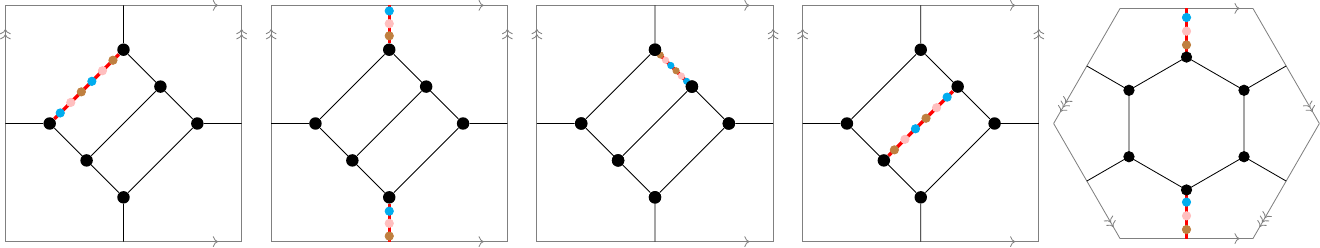}
        \caption{Subdivisions of $\K$ in $F_{11}$}
        \label{fig:K33InF11Cases}
    \end{figure}

    \noindent In Cases 1, 3, 4, and 5 of Figure \ref{fig:K33InF11Cases} $e^\ast$ is part of two distinct facial cycles bounding faces that are homeomorphic to a disk.

        \noindent We may assume without loss of generality that the edge $14$ goes through face $f_1$. By Lemma \ref{lem:JCT_consequenc} this cuts $f_1$ into two components $f'_1$ and $f''_1$ as pictured below. This forces the edge $25$ to go through face $f_2$, cutting it into two components $f'_2$ and $f''_2$ as pictured below. It is now impossible to add the final edge $36$ as the vertex $3$ is part of the facial cycles of $f''_1$ and $f''_2$, whereas the vertex $6$ is part of the facial cycles of $f'_1$ and $f'_2$. Hence, these embeddings of a subdivision of $\K$ cannot be extended to an embedding of $F_{11}$.
        \begin{figure}[H]
            \centering
            \includegraphics[height=2cm]{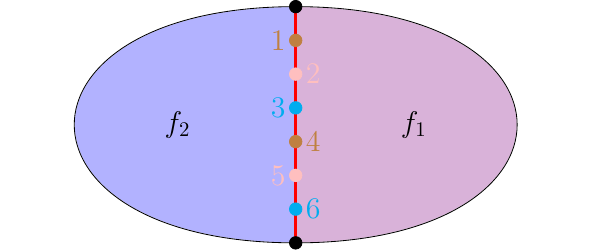}
            \includegraphics[height=2cm]{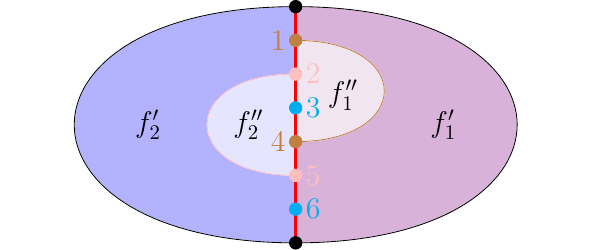}
            \label{fig:F11EmbeddingsBad2}
        \end{figure}

        \noindent Case 2 in Figure \ref{fig:K33InF11Cases} differs from the others as $e^\ast$ is part of only one facial cycle. 
        Note that cutting the torus along the highlighted edges of $\K$ gives a cylinder.
        \begin{center}
            \includegraphics[height=2cm]{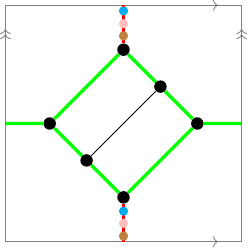}
        \end{center}

        \noindent We now want to connect vertices $2$ and $5$ by an edge. Note that we may view a cylinder as a disk with part of the boundary identified along $e^\ast$. In particular, there are two copies of $2$, namely $2'$ and $2''$ on the boundary of the disk, and two copies of $5$, namely $5'$ and $5''$. Note that by symmetries of the cylinder, a path from $2'$ to $5'$ is equivalent to a path from $2''$ to $5''$. Similarly, a path from $2'$ to $5''$ is equivalent to one from $2''$ to $5'$.
        \begin{center}
            \includegraphics[height=2.5cm]{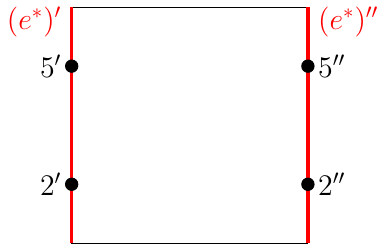}
        \end{center}

        \noindent By Lemma \ref{lem:JCT_consequenc}, a simple continuous path $p$ from $2'$ to $5'$ separates the interior of the disk into two components. One component has boundary $(e^\ast)'_2 \cup p$, while the other has boundary $(e^\ast)'_1 \cup r_1 \cup (e^\ast)'' \cup r_2 \cup (e^\ast)'_3 \cup p$. Similarly, a simple continuous path $q$ from $2'$ to $5''$ separates the interior of the disk into two components. One component has boundary $(e^\ast)'_2 \cup (e^\ast)'_1 \cup r_1 \cup (e^\ast)''_1 \cup q$, while the other has boundary $(e^\ast)''_2 \cup (e^\ast)''_3 \cup r_2 \cup (e^\ast)'_3 \cup q$. These components then correspond to components of the cylinder built from the usual quotient map.
        \begin{center}
            \includegraphics[height=2.5cm]{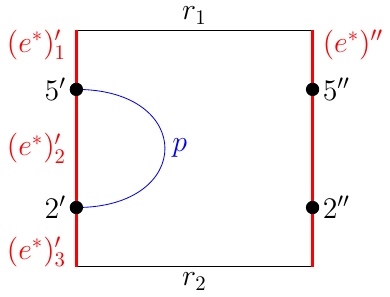} \hspace{.5cm}
            \includegraphics[height=2.5cm]{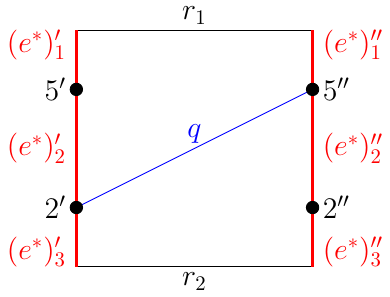}
        \end{center}

        \noindent 
        As just established, adding the edge $25$ cuts the cylinder into two components in one of two ways. We quickly see that in one of the cases, we cannot add the edge $36$ after having added the edge $14$ in either of the two possible ways, as shown below.
        \begin{center}
            \includegraphics[height=2cm]{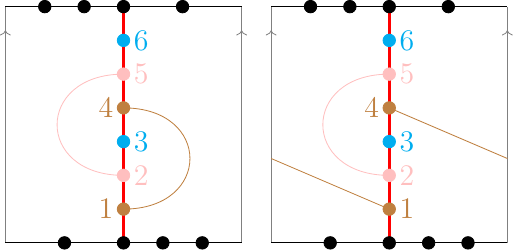}
        \end{center}

    \noindent In the other case we similarly have two ways of drawing the edge $14$. Moreover, once the edge $14$ is drawn, there are two ways of drawing the edge $36$. The four possible drawings are shown below.
    \begin{center}
        \includegraphics[height=2cm]{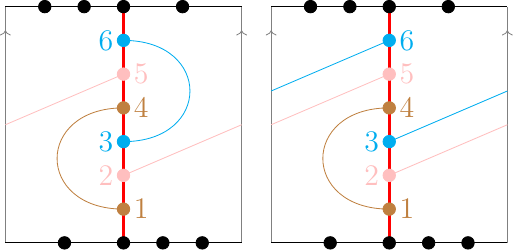}
        \includegraphics[height=2cm]{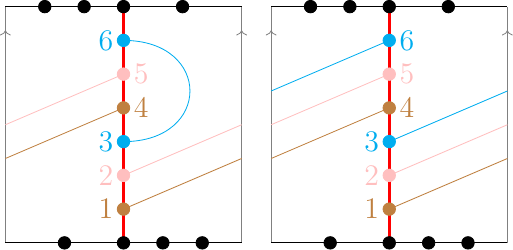}
    \end{center}

    \noindent Correspondingly, we have shown that all embeddings of $F_{11}$ are equivalent to one of the four pictured below.
    \begin{figure}[H]
        \centering
        \includegraphics[height=3cm]{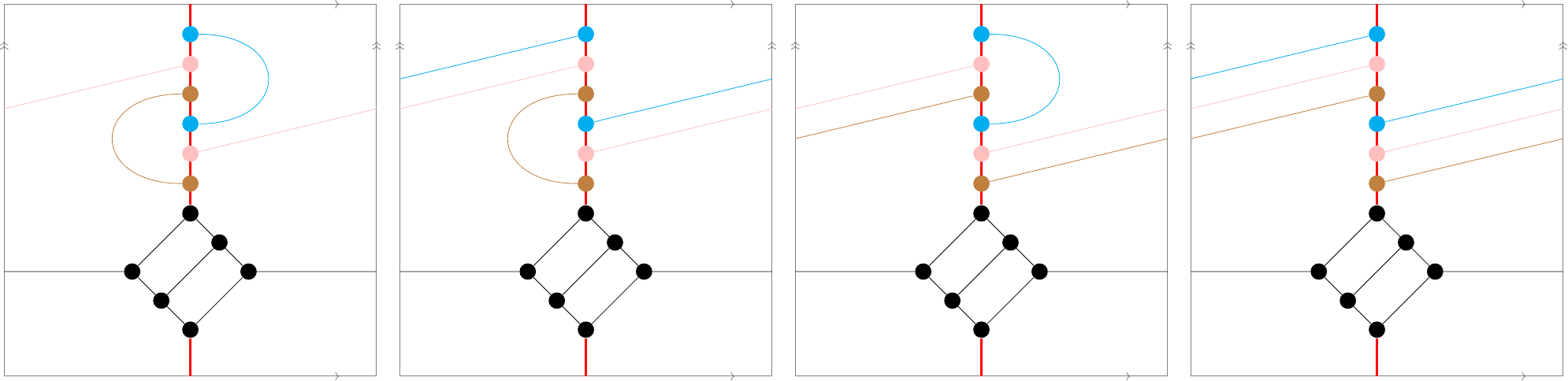}
        \caption{Extensions of $\K$ embeddings to $F_{11}$ embeddings}
        \label{fig:4ExtensionsOfF11}
    \end{figure}

    We will now show that each of the four embeddings of $F_{11}$ into the torus in Figure \ref{fig:4ExtensionsOfF11} is equivalent to one of the two embeddings in Figure \ref{fig:TwoF11Embeddings}. To this end, we will decompose each of the two embeddings in Figure \ref{fig:TwoF11Embeddings} into their facial $n$-gons and describe how to glue them together to obtain a torus. Then, we will show that each of the four embeddings in \ref{fig:4ExtensionsOfF11} can be decomposed and glued together in one of those two ways.

    \noindent Consider the four squares and two decagons below. Note that each vertex appears three times, and each edge appears twice. Identifying vertices with the same label and edges between the same pair of vertices, we obtain the 5-embedding of $F_{11}$ from Figure \ref{fig:TwoF11Embeddings}.
    \begin{figure}[H]
        \centering
        \includegraphics[height=2cm]{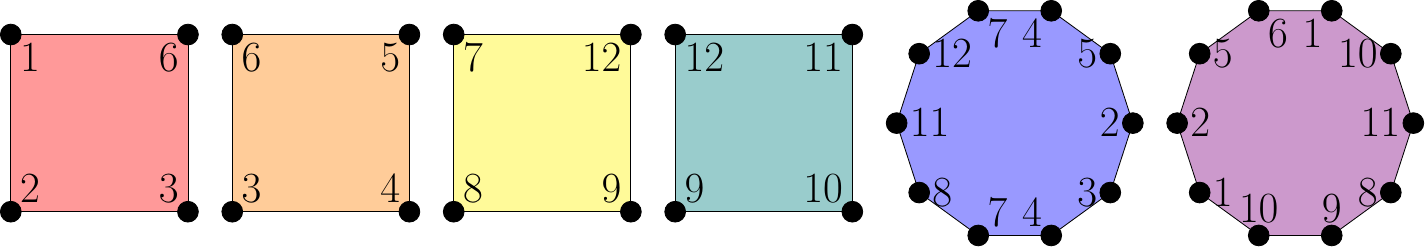} \hspace{.25cm}
        \includegraphics[height=3cm]{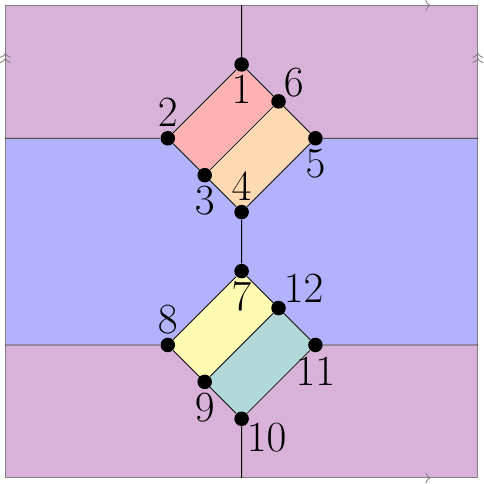}
        \caption{Decomposing the 5-embedding of $F_{11}$}
        \label{fig:5EmbF11Decomp}
    \end{figure}

    We observe that we can decompose the embeddings of $F_{11}$ in Figure \ref{fig:4ExtensionsOfF11} into four squares and two decagons.
    Moreover, we can label the vertices of the first and fourth embedding and decompose the torus as in Figure \ref{fig:5EmbF11Decomp}. This implies the existence of a homeomorphism of the torus that realizes an equivalence of the embedding from this case and the known 5-embedding of $F_{11}$.
        \begin{center}
            \includegraphics[height=3cm]{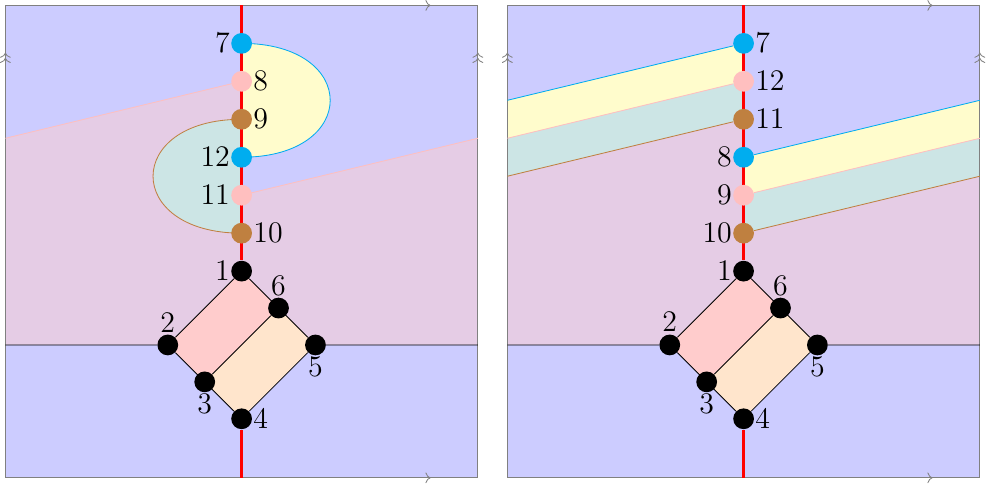}
        \end{center}

    Next, consider the four squares and two decagons below. Note that each vertex appears three times, and each edge appears twice. Identifying vertices with the same label and edges between the same pair of vertices, we obtain the 3-embedding of $F_{11}$ from Figure \ref{fig:TwoF11Embeddings}.
    \begin{figure}[H]
        \centering
        \includegraphics[height=2cm]{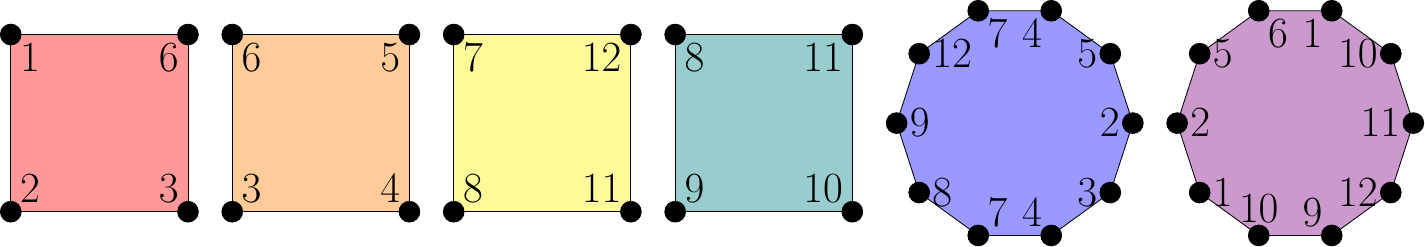} \hspace{.25cm}
        \includegraphics[height=3cm]{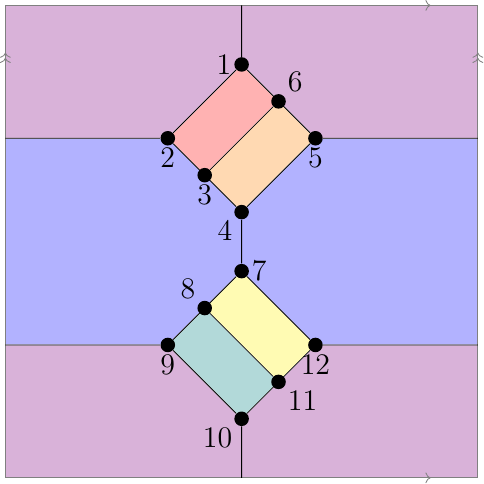}
        \caption{Decomposing the 3-embedding of $F_{11}$}
        \label{fig:3EmbF11Decomp}
    \end{figure}

    Similarly to above, we can label the vertices of the second and third embedding in Figure \ref{fig:4ExtensionsOfF11} and decompose the torus as in Figure \ref{fig:3EmbF11Decomp}. This implies the existence of a homeomorphism of the torus that realizes an equivalence of the embedding from this case and the known 3-embedding of $F_{11}$.
        \begin{center}
            \includegraphics[height=3cm]{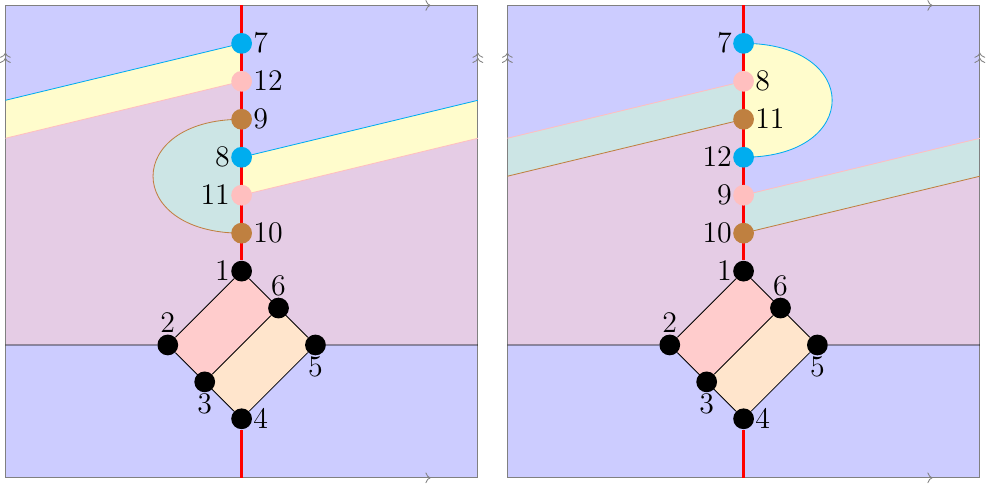}
        \end{center}

    \noindent We can now conclude that there are at most two inequivalent embeddings of $F_{11}$ into the torus, namely the ones depicted in Figure \ref{fig:TwoF11Embeddings}. The remark preceding the proof of Lemma \ref{lem:F11TorusEmbeddings} implies that these two embeddings are inequivalent, finishing the proof.   
    
\end{proof}

\section{Toroidal embeddings of \texorpdfstring{$F_{12}$}{F\_{12}}}
\label{sec:F12Lemma}

\setcounter{claim}{0}

\begin{lemma}
\label{lem:F12TorusEmbeddings}
    There are exactly four inequivalent unlabelled embeddings of $F_{12}$ into the torus.
\end{lemma}

\begin{figure}[H]
    \centering
    \includegraphics[height=2cm]{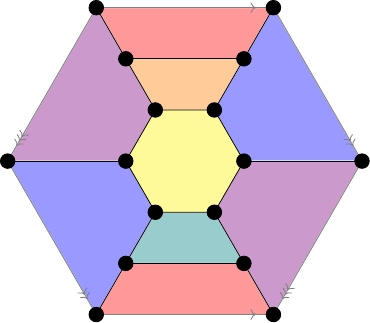} \hspace{.25cm}
    \includegraphics[height=2cm]{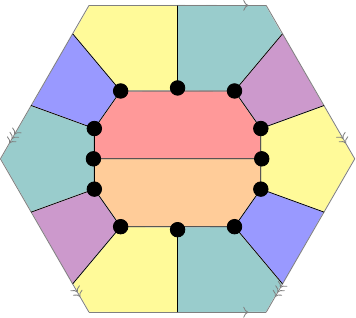} \hspace{.375cm}
    \includegraphics[height=2cm]{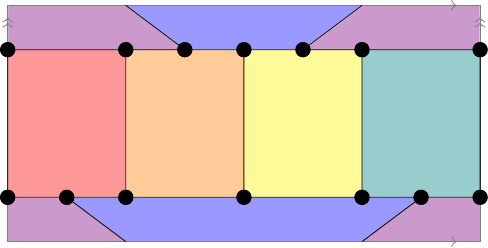} \hspace{.5cm}
    \includegraphics[height=2cm]{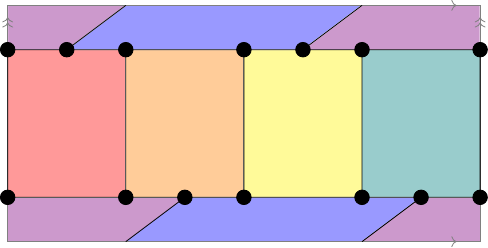}
    \caption{The four embeddings of $F_{12}$ into the torus}
    \label{fig:FourF12Embeddings}
\end{figure}
\begin{remark}
    \text{}
    \begin{enumerate}
        \item The first embedding from the left has two facial cycles of length four, two facial cycles of length six, and two facial cycles of length eight. In particular, it is the only embedding with a facial $6$-cycle and thus cannot be equivalent to any of the other embeddings. We will refer to it as the \textit{$6$-cycle embedding}.
        \item The second embedding from the left has two facial cycles of length four, and four facial cycles of length seven. In particular, it is the only embedding with a facial $7$-cycle and thus cannot be equivalent to any of the other embeddings. We will refer to it as the \textit{$7$-cycle embedding}.
        \item The remaining two embeddings both have four facial cycles of length five, and two facial cycles of length eight. To distinguish between the two, we observe that both embeddings contain a cylinder of $5$-cycles, and each boundary component of that cylinder has two vertices of degree two. In one of the embeddings, the vertices of degree two in one boundary component belong to adjacent $5$-cycles, whereas in the other they do not. More precisely, the distance between degree two vertices in a boundary component is two in one of the embeddings, and three in the other one. Thus, the two embeddings cannot be equivalent. We will refer to the first of these two embeddings as the \textit{2-embedding} and the other one as the \textit{3-embedding}.
        \begin{center}
            \includegraphics[height=1.5cm]{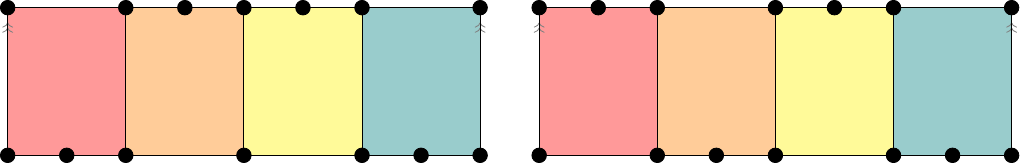}
        \end{center}
    \end{enumerate}
\end{remark}

\begin{proof}
    We proceed similarly as in the proof of Lemma \ref{lem:F11TorusEmbeddings}. Since $F_{12}$ is a cubic nonprojective graph, it is nonplanar and contains a subgraph homeomorphic to $K_{3,3}$. Furthermore, any subgraph homeomorphic to $\K$ is embedded into the torus in one of the ways pictured in Figure \ref{fig:K33InTorus}. To classify the embeddings of $F_{12}$, we fix a subgraph $H$ homeomorphic to $\K$ as shown below.
    \begin{figure}[H]
        \centering
        \includegraphics[height=2cm]{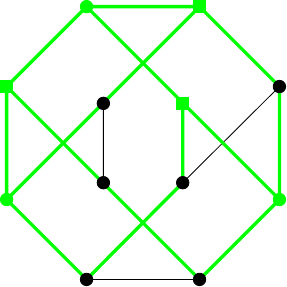} \hspace{.5cm}
        \includegraphics[height=2cm]{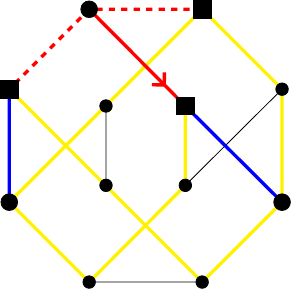}
    \end{figure}
    To build this subdivision of $\K$, note that the directed red edge $\vec e$ starts at the unique vertex of $\K$ none of whose incident edges get subdivided. The remaining edges incident to the endpoint of $\vec e$ get subdivided once. In the unique $4$-cycle of this subdivision, choose a pair of disjoint edges and subdivide each of these twice. In Figure \ref{fig:F12Mobius}, we have coloured the other two edges of this $4$-cycle in blue.
    Once we have the pictured subdivision of $\K$, there is a unique way to add the remaining edges to obtain $F_{12}$. In particular, $b$ has to be part of a $4$-cycle that includes a blue edge. Thus, $bb' \in E(F_{12})$. Similarly, $c$ has to part of a $4$-cycle that includes a blue edge, implying that $cc' \in E(F_{12})$. This forces $aa'$ to be the final edge.
    \begin{figure}[H]
        \centering
        \includegraphics[height=2cm]{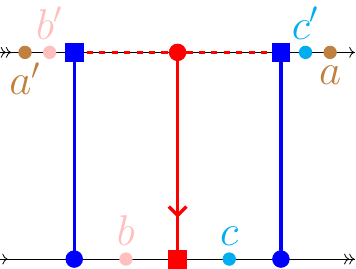} \hspace{1cm}
        \includegraphics[height=2.25cm]{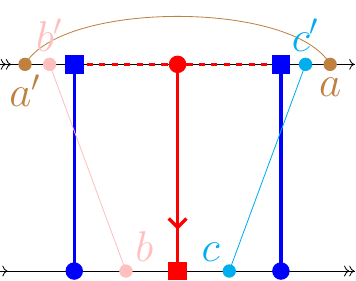}
        \caption{Building $F_{12}$ from $\K$}
        \label{fig:F12Mobius}
    \end{figure}

    \noindent By Lemma \ref{lem:ClassDirEdgeK33}, there are six inequivalent directed edged embeddings $(\K, \vec e) \hookrightarrow T^2$. Note that picking the directed edge forces the location of the dashed red edges in Figure \ref{fig:F12Mobius}. Then there are two ways to choose the blue edges (once the first one is picked, the second is forced since they have to be disjoint). Thus, decorating the directed edged embeddings with the two blue edges $e_1, e_2$, we get twelve \textit{blue directed edged embeddings}, $(\K, \vec e, \{e_1, e_2\}) \hookrightarrow T^2$. As described above, this completely determines $F_{12}$. Hence, we analyse the twelve cases below and check into which the remaining three edges $aa', bb', cc'$ can be added.
    \begin{figure}[H]
        \centering
        \includegraphics[height=4.85cm]{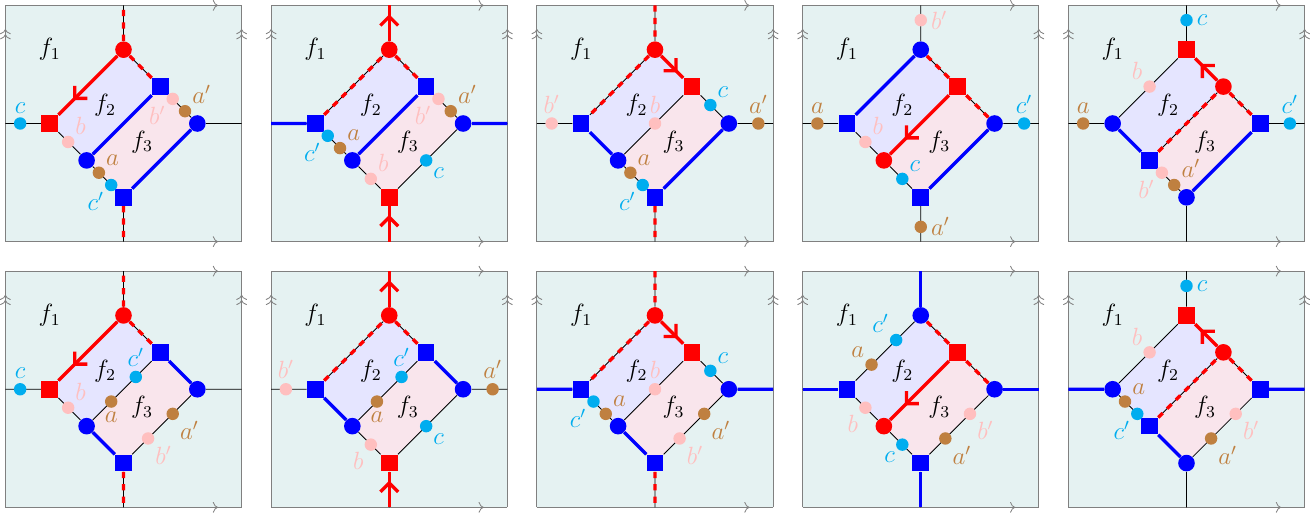}
        \includegraphics[height=4.85cm]{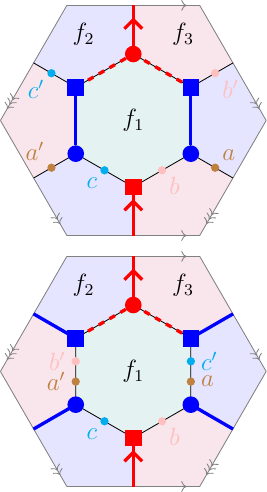}
        \caption{Subdivisions of $\K$ in $F_{12}$}
        \label{fig:K33InF12Cases}
    \end{figure}

    \noindent The cases are labelled 1a, 2a, etc. across the top row, and 1b, 2b, etc. across the bottom row.
    We immediately see the following:
    \begin{itemize}
        \item Case 1b: The edge $cc'$ cannot be added. Thus, this embedding does not extend to an embedding of $F_{12}$.
        \item Case 2b: The edge $aa'$ cannot be added. Thus, this embedding does not extend to an embedding of $F_{12}$.
        \item Case 3a: The edge $bb'$ cannot be added. Thus, this embedding does not extend to an embedding of $F_{12}$.
    \end{itemize}

    \noindent Note that since all embeddings of $\K$ into the torus are cellular, each face is homeomorphic to a disk. Thus, if two vertices are on the boundary of a common face and neither of them appears more than once, there is a unique way to add an edge between the vertices up to equivalence.

    \noindent We now consider the cases that will not result in embeddings of $F_{12}$:
    \begin{itemize}
        \item Case 3b: $f_1$ is the unique face such that the vertices $a$ and $a'$ appear on its boundary, and each appears exactly once. Thus, there is a unique way up to equivalence to draw the edge $aa'$. However, now there is no face such that vertices $c$ and $c'$ appear on its boundary. Hence, we cannot extend the given embedding of a subdivision of $\K$ to an embedding of $F_{12}$.
            \begin{center}
                \includegraphics[height=3cm]{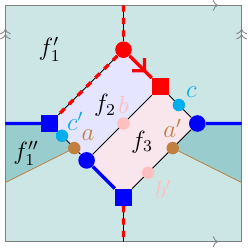}
            \end{center}

        \item Case 5a: $f_1$ is the unique face such that the vertices $b$ and $b'$ appear on its boundary, and each appears exactly once. Thus, there is a unique way to draw the edge $bb'$. However, now there is no face such that vertices $a$ and $a'$ appear on its boundary. Hence, we cannot extend the given embedding of a subdivision of $\K$ to an embedding of $F_{12}$.
            \begin{figure}[H]
                \centering
                \includegraphics[height=3cm]{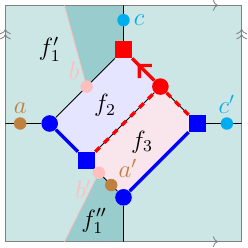}
                \label{fig:F12Case5a}
            \end{figure}

        \item Case 5b: Similar to the previous two cases, $f_1$ is the unique face such that vertices $b$ and $b'$ appear on its boundary, and each appears exactly once. However, now there is no face such that vertices $a$ and $a'$ appear on its boundary. Hence, we cannot extend the given embedding of a subdivision of $\K$ to an embedding of $F_{12}$.
            \begin{figure}[H]
                \centering
                \includegraphics[height=3cm]{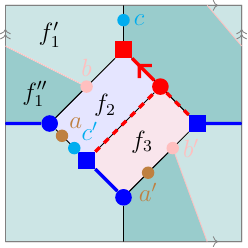}
                \label{fig:F12Case5b}
            \end{figure}
    \end{itemize}

    \noindent Finally, we proceed to the remaining cases, each of which gives at least one embedding of $F_{12}$ into the torus.
    \begin{itemize}
        \item Case 1a: $f_1$ is the unique face such that the vertices $b$ and $ b'$ appear on its boundary, and each appears exactly once. Thus, there is a unique way to draw the edge $bb'$. This then forces the edge $cc'$ to go through face $f'_1$. Finally, for the edge $aa'$ there are two choices, giving the two embeddings of $F_{12}$ pictured below.
            \begin{figure}[H]
                \centering
                \includegraphics[height=3cm]{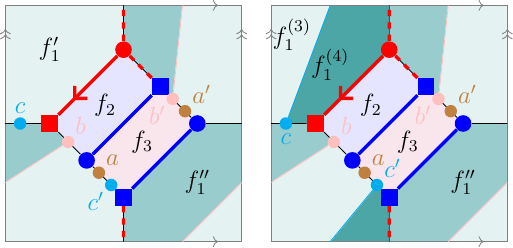} \hspace{.5cm}
                \includegraphics[height=3cm]{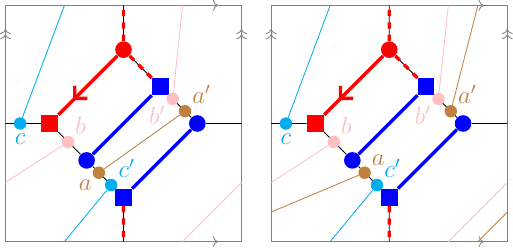}
                \label{fig:F12Case1a}
            \end{figure}

        \item Case 2a: $f_1$ is the unique face such that the vertices $a$ and $a'$ appear on its boundary, and each appears eactly once. Moreover, it is the only face with vertices $c$ and $c'$ on its boundary, and each appears exactly once. Thus, there is at most one way to draw the edges $aa', cc'$. It is easy to check that it is indeed possible to draw both edges, as shown below. This then forces the edge $bb'$ to go through face $f_3$, and we get the embedding of $F_{12}$ pictured below.
            \begin{figure}[H]
                \centering
                \includegraphics[height=3cm]{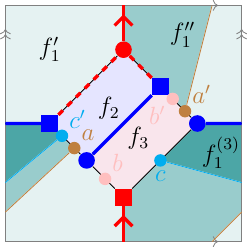} \hspace{1cm}
                \includegraphics[height=3cm]{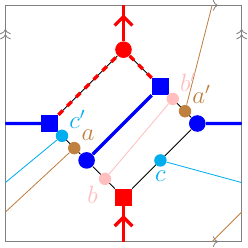}
                \label{fig:F12Case2a}
            \end{figure}

        \item Case 4a: $f_1$ is the unique face such that the vertices $c$ and $c'$ appear on its boundary, and vertex $c'$ appears twice. Thus, there are two ways to draw the edge $cc'$. It is quick to see that one of these two options makes drawing the edge $bb'$ impossible (left picture). Thus, we choose the other option (middle picture). Then there exists a unique way to draw the edge $bb'$ as $f'_1$ is the unique face such that both vertices appear on its boundary, and each appears exactly once. Finally, for the edge $aa'$ there are two choices. Hence, we get the two embeddings of $F_{12}$ pictured below.
            \begin{figure}[H]
                \centering
                \includegraphics[height=3cm]{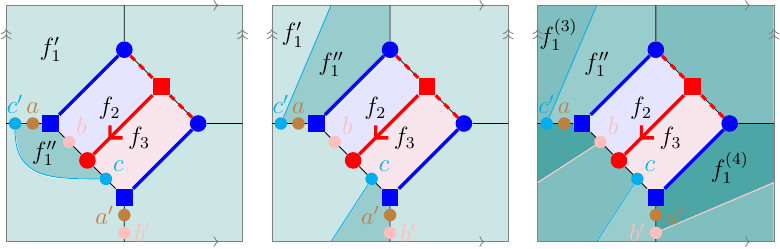} \hspace{1cm}
                \includegraphics[height=3cm]{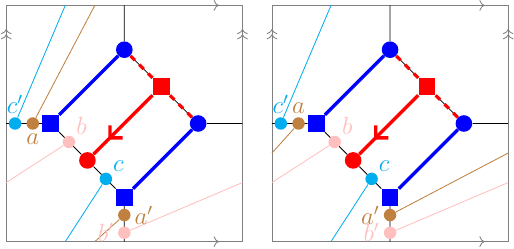}
                \label{fig:F12Case4a}
            \end{figure}

        \item Case 4b: Each of the vertices $a, a', b, b', c,$ and $c'$ appears only on the boundary of face $f_1$, and each appears exactly once. Thus, there is at most one way to draw all three edges $aa', bb', cc'$, pictured below.
            \begin{figure}[H]
                \centering
                \includegraphics[height=2.75cm]{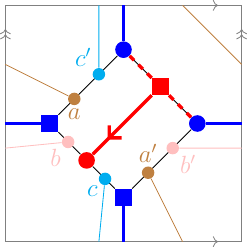} 
                \label{fig:F12Case4b}
            \end{figure}

        \item Case 6a: $f_3$ is the unique face such that the vertices $b$ and $b'$ appear on its boundary, and $f_2$  is the unique face such that the vertices $c$ and $c'$ appear on its boundary. Moreover, each appears exactly once. Thus, there is a unique way to draw the edges $bb'$ and $cc'$. Finally, for the edge $aa'$ there are two choices, giving the two embeddings of $F_{12}$ pictured below.
            \begin{figure}[H]
                \centering
                \includegraphics[height=2.75cm]{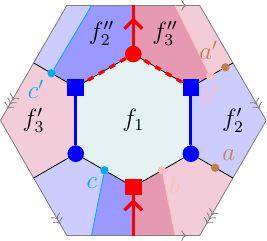} \hspace{.5cm}
                \includegraphics[height=2.75cm]{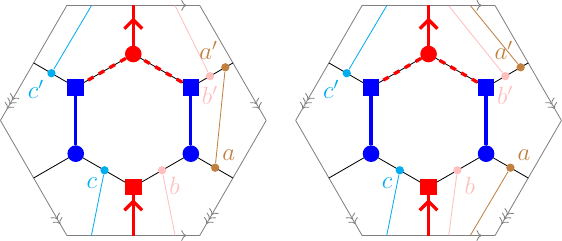}
                \label{fig:F12Case7a}
            \end{figure}

        \item Case 6b: $f_1$ is the unique face such that the vertices $a$ and $a'$ appear on its boundary, and each appears exactly once. Thus, there is a unique way to draw the edge $aa'$. Now, $f_3$ is the unique face with vertices $b$ and $b'$ on its boundary, and $f_2$ is the unique face with vertices $c$ and $c'$ on its boundary. Moreover, each vertex appears exactly once on those boundaries. Thus, we get the embedding of $F_{12}$ shown below.
            \begin{figure}[H]
                \centering
                \includegraphics[height=3cm]{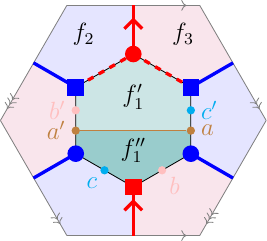} \hspace{1cm}
                \includegraphics[height=3cm]{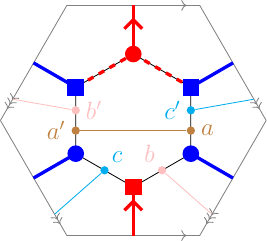}
                \label{fig:F12Case7b}
            \end{figure}
    \end{itemize}

    Using the same technique as in Section \ref{sec:F11Lemma}, we will now show that each of these nine embeddings of $F_{12}$ into the torus is equivalent to one of the four embeddings in Figure \ref{fig:FourF12Embeddings}. 

    \noindent Consider the two squares, two hexagons, and two octagons below. Note that each vertex appears three times, and each edge appears twice. Identifying vertices with the same label and edges between the same pair of vertices, we obtain the $6$-cycle embedding of $F_{12}$ from Figure \ref{fig:FourF12Embeddings}.
    \begin{figure}[H]
        \centering
        \includegraphics[height=1.75cm]{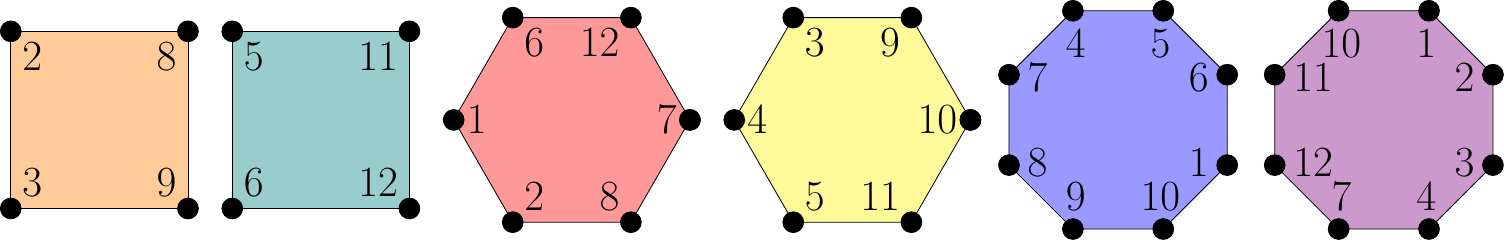} \hspace{.25cm}
        \includegraphics[height=3cm]{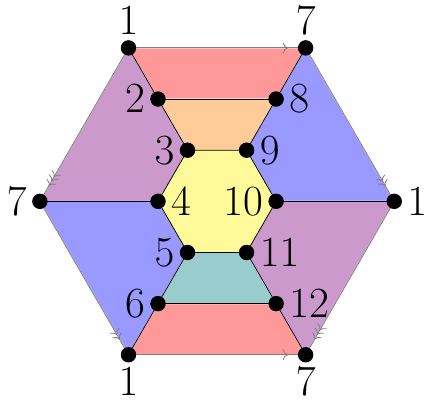}
        \caption{Decomposing the $6$-cycle embedding of $F_{12}$}
        \label{fig:6CycleF12Decomp}
    \end{figure}

    We now observe that we can decompose the embedding of $F_{12}$ resulting from Case 2a into two squares, two hexagons, and two octagons. Moreover, we can label the vertices of this embedding and decompose the torus as in Figure \ref{fig:6CycleF12Decomp}. This implies the existence of a homeomorphism of the torus that realizes an equivalence of the embedding from this case and the known $6$-cycle embedding of $F_{12}$.
        \begin{center}
            \includegraphics[height=3cm]{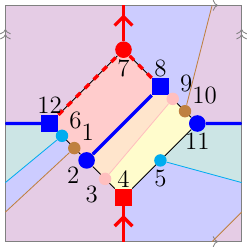}
        \end{center}

    Next, consider the two squares and four heptagons below. Identifying vertices with the same label and edges between the same pair of vertices, we obtain the $7$-cycle embedding of $F_{12}$ from Figure \ref{fig:FourF12Embeddings}.
    \begin{figure}[H]
        \centering
        \includegraphics[height=1.75cm]{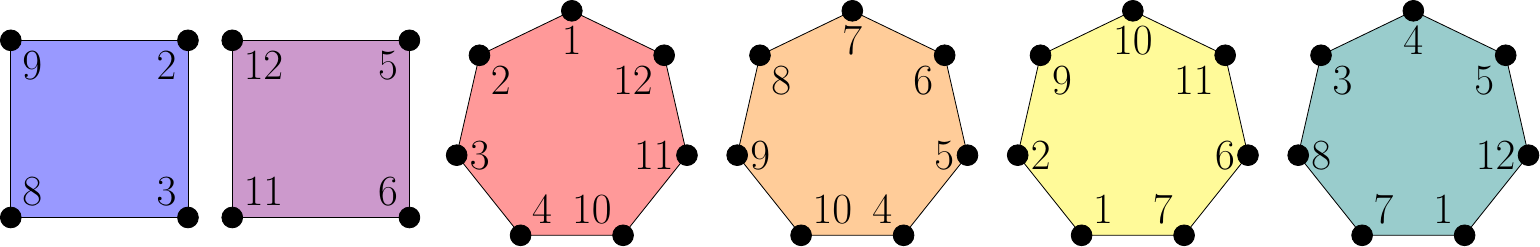} \hspace{.25cm}
        \includegraphics[height=3cm]{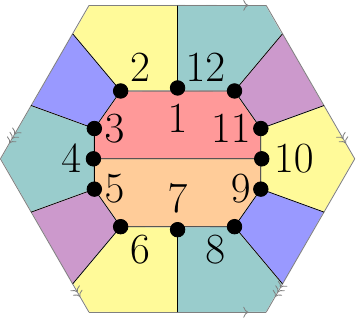}
        \caption{Decomposing the $7$-cycle embedding of $F_{12}$}
        \label{fig:7CycleF12Decomp}
    \end{figure}

    We observe that we can decompose the two embeddings of $F_{12}$ resulting from Cases 4b and 6b into two squares, and four heptagons. Moreover, we can label the vertices of this embedding and decompose the torus as in Figure \ref{fig:7CycleF12Decomp}. Hence, the embeddings from these cases are equivalent to the known $7$-cycle embedding of $F_{12}$.
        \begin{center}
            \includegraphics[height=3cm]{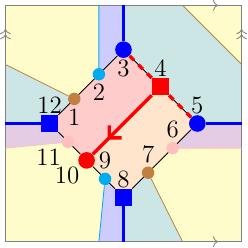}
            \includegraphics[height=3cm]{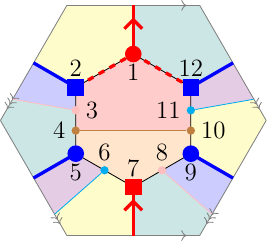}
        \end{center}

    Consider the four pentagons and two octagons. Identifying vertices with the same label and edges between the same pair of vertices, we obtain the 2-embedding of $F_{12}$ from Figure \ref{fig:FourF12Embeddings}.
    \begin{figure}[H]
        \centering
        \includegraphics[height=1.65cm]{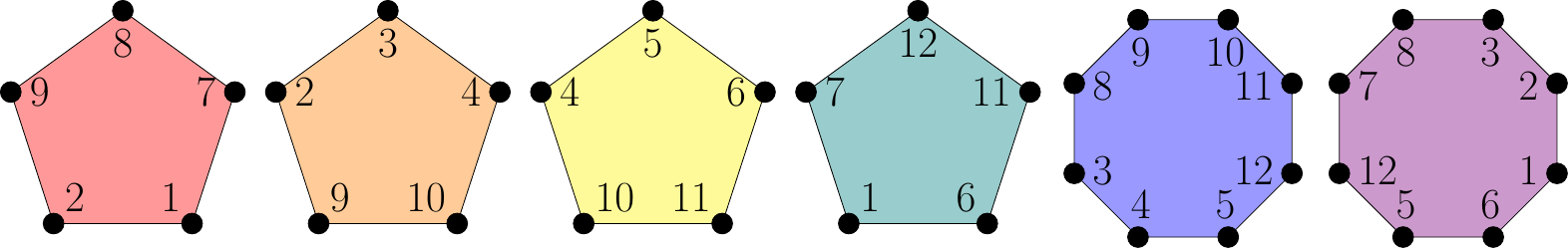} \hspace{.25cm}
        \includegraphics[height=1.75cm]{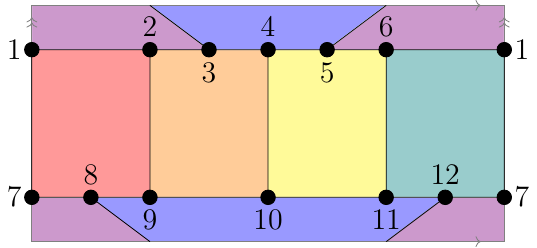}
        \caption{Decomposing the 2-embedding of $F_{12}$}
        \label{fig:11F12Decomp}
    \end{figure}

    We note that we can decompose the two embeddings of $F_{12}$ resulting from Case 1a into five pentagons, and two octagons. Moreover, we can label the vertices of the two embeddings resulting from Case 1a and decompose the torus as in Figure \ref{fig:11F12Decomp}. Hence, the embeddings from this case are equivalent to the known 2-embedding of $F_{12}$.
        \begin{center}
            \includegraphics[height=3cm]{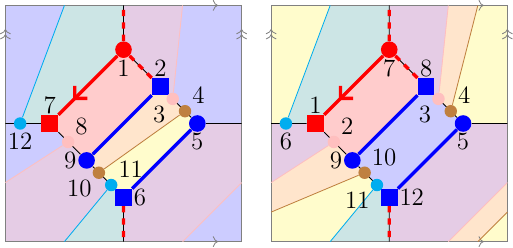}
        \end{center}

    Finally, we consider the four pentagons and two octagons below. Identifying vertices with the same label and edges between the same pair of vertices, we obtain the 3-embedding of $F_{12}$ from Figure \ref{fig:FourF12Embeddings}.
    \begin{figure}[H]
        \centering
        \includegraphics[height=1.65cm]{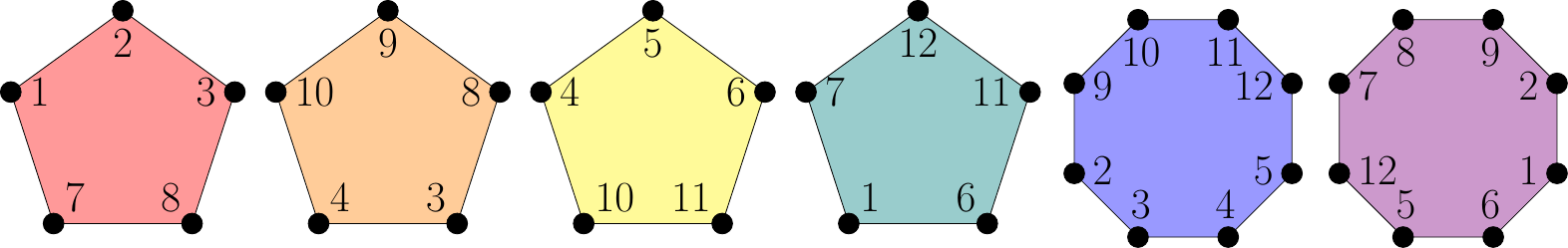} \hspace{.25cm}
        \includegraphics[height=1.75cm]{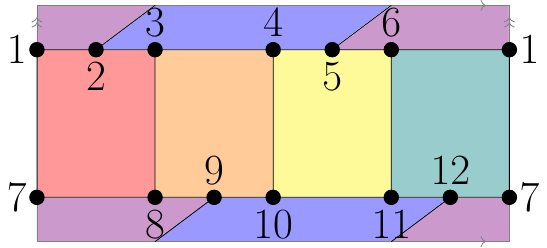}
        \caption{Decomposing the 3-embedding of $F_{12}$}
        \label{fig:12F12Decomp}
    \end{figure}

    We can decompose the four embeddings of $F_{12}$ resulting from Case 4a and 6a into five pentagons, and two octagons. Moreover, we can label the vertices of these four embeddings and decompose the torus as in Figure \ref{fig:12F12Decomp}. Hence, the embeddings from these cases are equivalent to the known 3-embedding of $F_{12}$.
    \begin{center}
        \includegraphics[height=3cm]{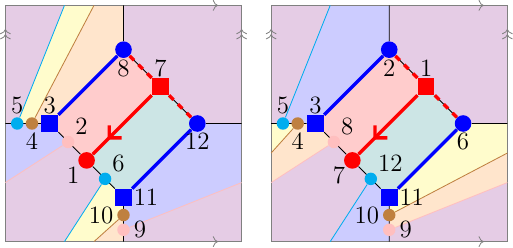}
        \includegraphics[height=3cm]{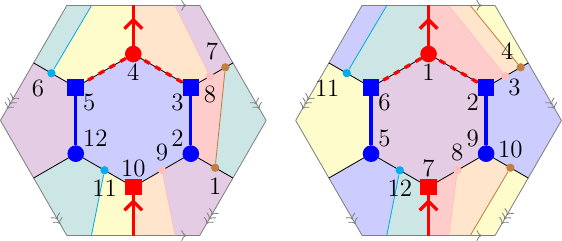}
    \end{center}

    \noindent We can now conclude that there are at most four inequivalent embeddings of $F_{12}$ into the torus, namely the ones depicted in Figure \ref{fig:FourF12Embeddings}. The remarks preceding the proof of Lemma \ref{lem:F12TorusEmbeddings} imply that these four embeddings are pairwise inequivalent, finishing the proof.   
    
\end{proof}

\section{Toroidal embeddings of \texorpdfstring{$F_{13}$}{F\_{13}}}
\label{sec:F13Lemma}

\setcounter{claim}{0}

\begin{lemma}
\label{lem:F13TorusEmbeddings}
    There are exactly two inequivalent unlabelled embedding of $F_{13}$ into the torus.
\end{lemma}

\begin{figure}[H]
    \centering
    \includegraphics[height=3cm]{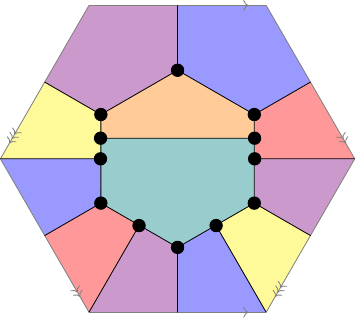} \hspace{.5cm}
    \includegraphics[height=3cm]{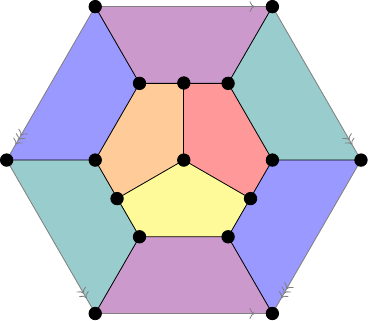}
    \caption{The two embeddings of $F_{13}$ into the torus}
    \label{fig:TwoF13Embeddings}
\end{figure}

\begin{remark}
    \text{}
    \begin{enumerate}
        \item The embedding on the left has three facial cycles of length five, two facial cycles of length six, and one facial cycles of length nine. In particular, it is the only embedding with a facial $9$-cycle and thus cannot be equivalent to the other embedding. We will refer to it as the \textit{$9$-cycle embedding}.
        \item The embedding on the right has three facial cycles of length five, and three facial cycles of length seven. In particular, it is the only embedding with a facial $7$-cycle. We will refer to it as the \textit{$7$-cycle embedding}.
    \end{enumerate}
\end{remark}

\begin{proof}
    Similar to the other graphs we have analysed in Sections \ref{sec:F11Lemma} and \ref{sec:F12Lemma}, we fix a subgraph $H$ of $F_{13}$ homeomorphic to $\K$ as shown below.
    \begin{figure}[H]
        \centering
        \includegraphics[height=2cm]{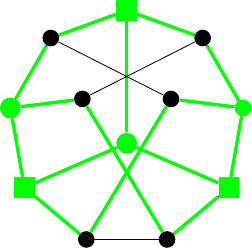} \hspace{.5cm}
        \includegraphics[height=2cm]{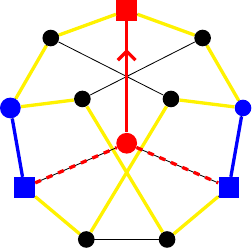}
    \end{figure}
    The subdivision of $\K$ is shown below, where $\K$ is drawn as the Möbius ladder on three rungs. Note that the unlabelled versions of the graphs on the left in Figures \ref{fig:F12Mobius} and \ref{fig:F13Mobius} are identical.
    Once we have the pictured subdivision of $\K$, there is a unique way to add the remaining edges to obtain $F_{13}$ by recalling that $F_{13}$ is cubic and has girth five.
    \begin{figure}[H]
        \centering
        \includegraphics[height=2cm]{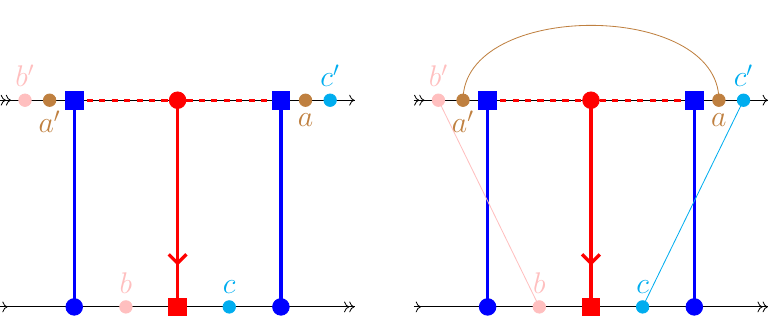}
        \caption{Building $F_{13}$ from $\K$}
        \label{fig:F13Mobius}
    \end{figure}

    \noindent In the same way as in the proof of Lemma \ref{lem:F12TorusEmbeddings} we get twelve \textit{blue directed edged embeddings} of the subdivision of $\K$.
    \begin{center}
        \includegraphics[height=2.25cm]{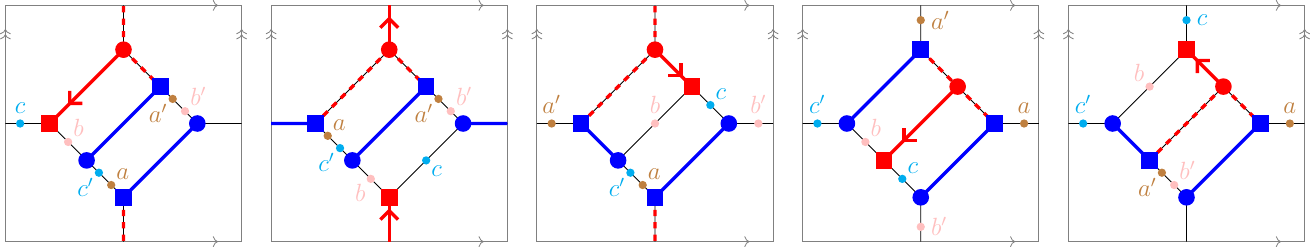}
        \includegraphics[height=2.25cm]{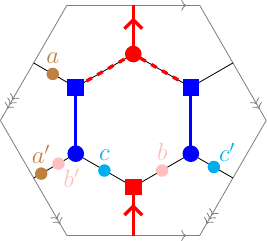}
    \end{center}
    \begin{figure}[H]
        \centering
        \includegraphics[height=2.25cm]{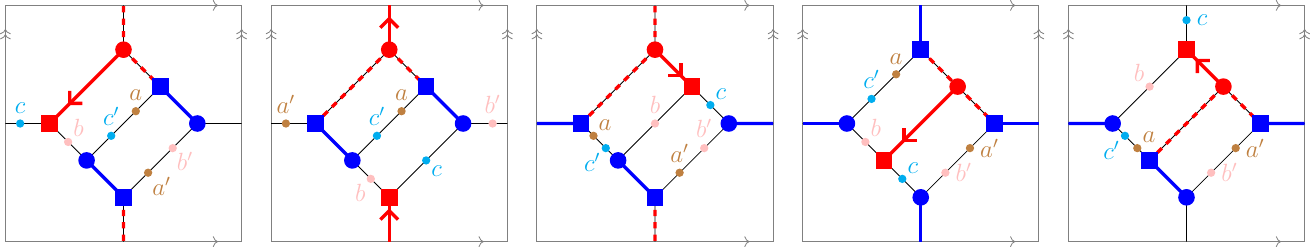}
        \includegraphics[height=2.25cm]{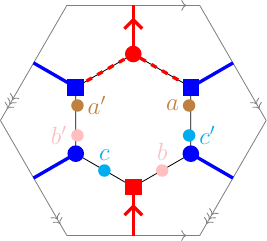}
        \caption{Subdivisions of $\K$ in $F_{13}$}
        \label{fig:K33InF13Cases}
    \end{figure}

    The cases are labelled 1a, 2a, etc. across the top row, and 1b, 2b, etc. across the bottom row.
    We observe that not only are the unlabelled subdivision of $\K$ from which we build $F_{12}$ and $F_{13}$ identical, but removing the edge $aa'$ from $F_{12}$ and $F_{13}$ gives homeomorphic graphs. Moreover, the labelled subdivisions from Figures \ref{fig:K33InF12Cases} and \ref{fig:K33InF13Cases} are closely related. To get from one to the other, one only has to swap vertices $a'$ and $b'$, and vertices $a$ and $c'$, respectively.
    As the case analysis here is similar to the case analysis in the proof of Lemma \ref{lem:F12TorusEmbeddings}, we omit the details. Cases 1a, 3b, and 4a each result in one embedding of $F_{13}$, while Cases 5a and 6b each result in two embeddings of $F_{13}$. See Figures \ref{fig:F13Cases1a5a6b} and \ref{fig:F13Cases3b4a} for pictures of these embeddings. 
    Using the same technique as in the previous sections, we will now show that each of these seven embeddings of $F_{13}$ into the torus is equivalent to one of the two embeddings in Figure \ref{fig:TwoF13Embeddings}. 

    \noindent Consider the three pentagons, two hexagons, and one nonagon below. Identifying vertices with the same label and edges between the same pair of vertices, we obtain the $9$-cycle embedding of $F_{13}$ from Figure \ref{fig:TwoF13Embeddings}.
    \begin{figure}[H]
        \centering
        \includegraphics[height=1.75cm]{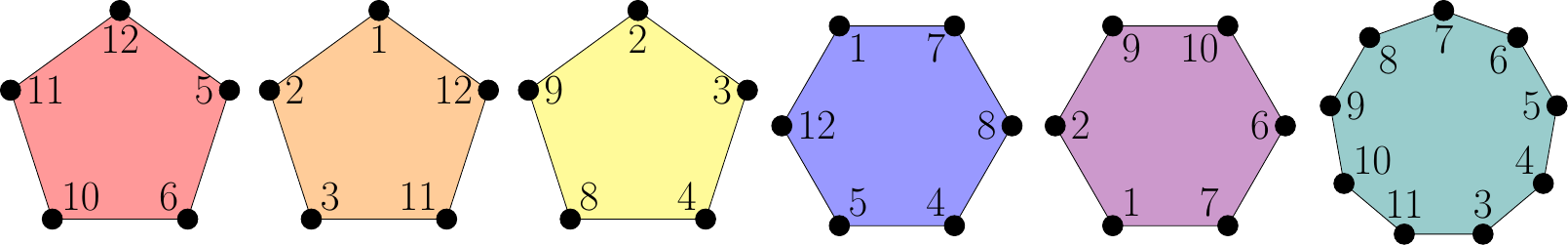} \hspace{.25cm}
        \includegraphics[height=2.75cm]{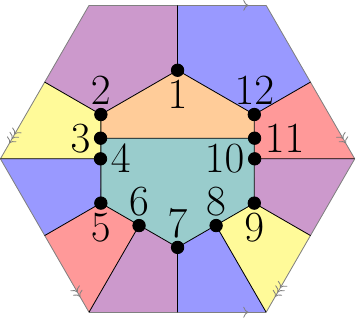}
        \caption{Decomposing the $9$-cycle embedding of $F_{13}$}
        \label{fig:9CycleF13Decomp}
    \end{figure}

    We observe that we can decompose the five embeddings of $F_{13}$ resulting from Cases 1a, 5a and 6b into three pentagons, two hexagons and one nonagon. 
    Moreover, we can label the vertices of these embeddings and decompose the torus as in Figure \ref{fig:9CycleF13Decomp}. This implies the existence of a homeomorphism of the torus that realizes an equivalence of the embedding from these cases and the known $6$-cycle embedding of $F_{12}$.
    \begin{figure}[H]
        \centering
        \includegraphics[height=2.75cm]{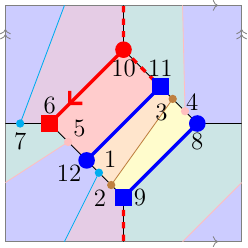}
        \includegraphics[height=2.75cm]{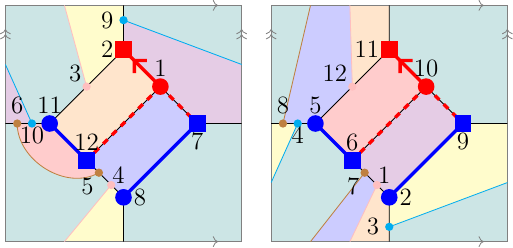}
        \includegraphics[height=2.75cm]{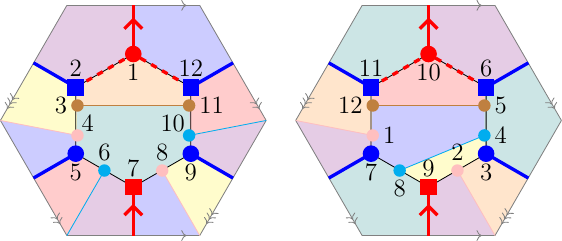}
        \caption{Embeddings of $F_{13}$ from Cases 1a, 5a, and 6a}
        \label{fig:F13Cases1a5a6b}
    \end{figure}

    Similarly, consider the three pentagons and three heptagons below. Note that each vertex appears three times, and each edge appears twice. Identifying vertices with the same label and edges between the same pair of vertices, we obtain the $7$-cycle embedding of $F_{13}$ from Figure \ref{fig:TwoF13Embeddings}.
    \begin{figure}[H]
        \centering
        \includegraphics[height=1.75cm]{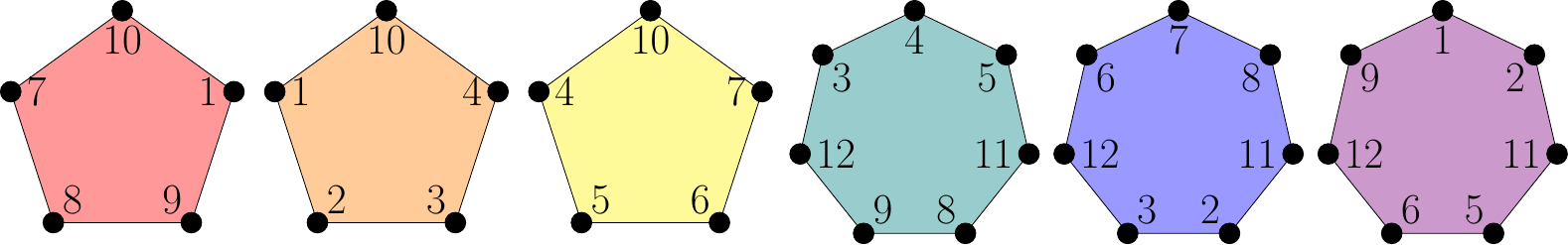} \hspace{.25cm}
        \includegraphics[height=2.75cm]{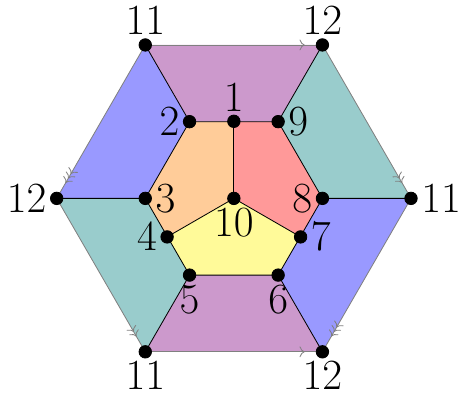}
        \caption{Decomposing the $7$-cycle embedding of $F_{13}$}
        \label{fig:7CycleF13Decomp}
    \end{figure}

    We can decompose the two embeddings of $F_{13}$ resulting from Cases 3b and 4a into three pentagons and four heptagons. Moreover, we can label the vertices of these embeddings and decompose the torus as in Figure \ref{fig:7CycleF13Decomp}. Hence, the embeddings from these cases are equivalent to the known $7$-cycle embedding of $F_{13}$.
        \begin{figure}[H]
            \centering
            \includegraphics[height=2.75cm]{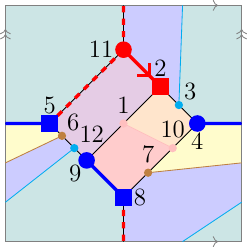}
            \includegraphics[height=2.75cm]{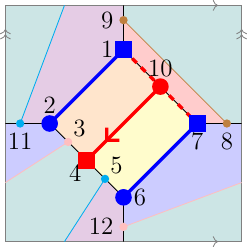}
            \caption{Embeddings of $F_{13}$ from Cases 3b and 4a}
            \label{fig:F13Cases3b4a}
        \end{figure}

    \noindent We can now conclude that there are at most two inequivalent embeddings of $F_{13}$ into the torus, namely the ones depicted in Figure \ref{fig:TwoF13Embeddings}. The remarks preceding the proof of Lemma \ref{lem:F13TorusEmbeddings} imply that these two embeddings are inequivalent, finishing the proof.   
    
\end{proof}

\section{Toroidal embeddings of \texorpdfstring{$F_{14}$}{F\_{14}}}
\label{sec:F14Lemma}

\setcounter{claim}{0}

\begin{lemma}
\label{lem:F14TorusEmbeddings}
    There are exactly two inequivalent unlabelled embeddings of $F_{14}$ into the torus.
\end{lemma}

\begin{figure}[H]
    \centering
    \includegraphics[height=3cm]{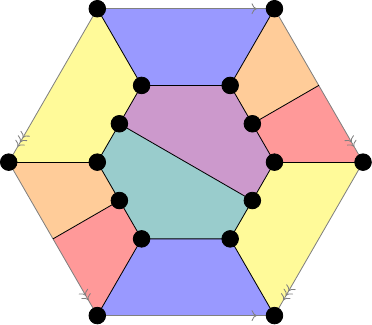} \hspace{.25cm}
    \includegraphics[height=3cm]{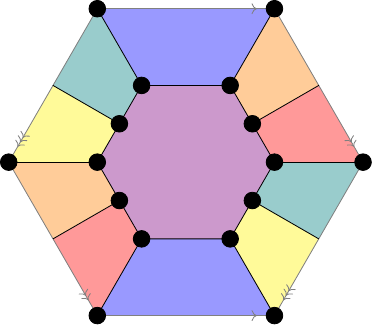}
    \caption{The two embeddings of $F_{14}$ into the torus}
    \label{fig:TwoF14Embeddings}
\end{figure}

\newpage
\begin{remark}
    \text{}
    \begin{enumerate}
        \item The embedding on the left has two facial cycles of length five, three facial cycles of length six, and one facial cycles of length eight. In particular, it is the only embedding with a facial $8$-cycle and thus cannot be equivalent to the other embedding. We will refer to it as the \textit{$8$-cycle embedding}.
        \item The embedding on the right has four facial cycles of length five, one facial cycle of length six, and one facial cycle of length ten. In particular, it is the only embedding with a facial $10$-cycle. We will refer to it as the \textit{$10$-cycle embedding}.
    \end{enumerate}
\end{remark}

\begin{proof}
    Similar to the previously analysed graphs, we fix a subgraph $H$ of $F_{14}$ homeomorphic to $\K$ as shown below.
    \begin{figure}[H]
        \centering
        \includegraphics[height=2cm]{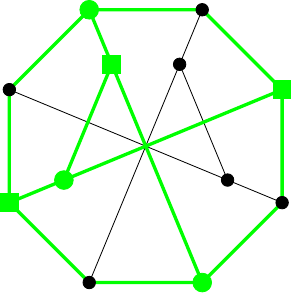} \hspace{.25cm}
        \includegraphics[height=2cm]{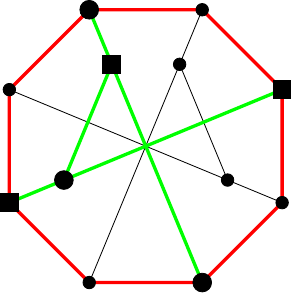}
    \end{figure}
    
    The subdivision of $\K$ is shown below, where $\K$ is drawn as the Möbius ladder on three rungs. To build $F_{14}$, choose one of the facial $4$-cycles, subdivide each of its edges once, and connect the new vertices of degree two on disjoint edges. The two new edges then get subdivided once, and the new vertices get connected.
    \begin{figure}[H]
        \centering
        \includegraphics[height=2.5cm]{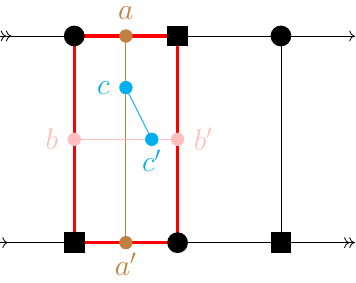} 
    \end{figure}

    \noindent By Lemma \ref{lem:ClassCycleK33}, there are five inequivalent cycled embeddings $(\K, C_4) \hookrightarrow T^2$. As described above, picking a $4$-cycle in $\K$ completely determines how to build $F_{14}$. Hence, we analyse the five cases below and check which ones can be extended to an embedding of $F_{14}$.
    \begin{figure}[H]
        \centering
        \includegraphics[width=\textwidth]{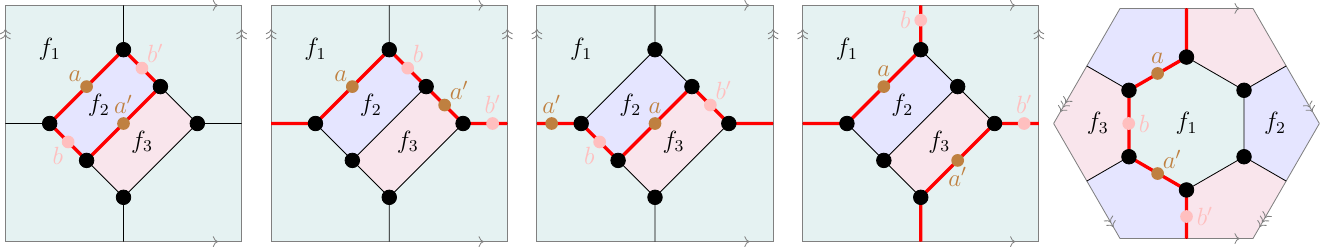}
        \caption{Subdivisions of $\K$ in $F_{14}$}
        \label{fig:K33InF14Cases}
    \end{figure}

    We immediately see that in Case 3 the edge $aa'$ cannot be added. Thus, this embedding does not extend to an embedding of $F_{14}$.

    \noindent Note that since all embeddings of $\K$ into the torus are cellular, each face is homeomorphic to a disk. Thus, if two vertices are on the boundary of a common face and neither of them appears more than once, there is a unique way to add an edge between the vertices.

    \begin{itemize}
        \item Case 1: $f_2$ is the unique face such that the vertices $a$ and $a'$ appear on its boundary, while face $f_1$ is the unique face such that vertices $b$ $b'$ appear on its boundary. This leads to vertex $c$ to be in face $f_2$, and vertex $c'$ to be in face $f_1$. Therefore, we cannot add the edge $cc'$ and the given embedding of a subdivision of $\K$ cannot be extended to an embedding of $F_{14}$. 
            \begin{figure}[H]
                \centering
                \includegraphics[height=3cm]{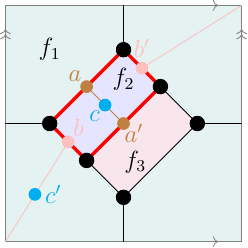} 
                \label{fig:F14Case1}
            \end{figure}

        \item Case 2: $f_1$ is the unique face such that the vertices $a$ and $a'$ appear on its boundary, and each vertex appears exactly once. Thus, there exists a unique way to draw the edge $aa'$. Now, $f'_1$ is the unique face such that the vertices $b$ and $b'$ appear on its boundary, and each vertex appears exactly once. So, there exists a unique way to draw the edge $bb'$. Finally, $f^{(3)}_1$ is the unique face such that vertices $c, c'$ appear on its boundary, and they each appear exactly once. Hence, we get the embedding of $F_{14}$ pictured below.
            \begin{figure}[H]
                \centering
                \includegraphics[height=3cm]{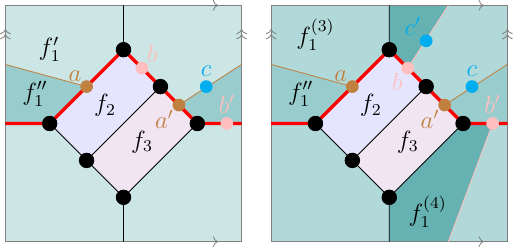} \hspace{.5cm}
                \includegraphics[height=3cm]{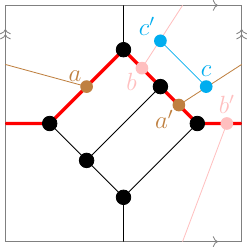}
                \label{fig:F14Case2}
            \end{figure}

        \item Case 4: Similar to Case 2, there is a unique way to draw the edges $aa'$. Now, the vertices $b$ and $b'$ both appear once on the boundary of face $f'_1$ and once on the boundary of face $f''_1$. Hence, there are two ways to draw the edge $bb'$. For each of the two choices, there is then a unique way to draw the edge $cc'$, giving the two embeddings of $F_{14}$ pictured below.
            \begin{figure}[H]
                \centering
                \includegraphics[height=3cm]{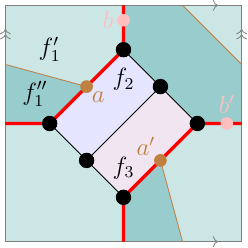} \hspace{.5cm}
                \includegraphics[height=3cm]{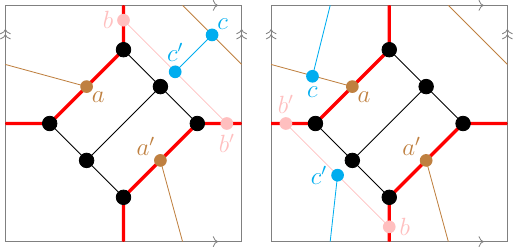}
                \label{fig:F14Case4}
            \end{figure}

        \item Case 5: $f_3$ is the unique face such that the vertices $b, b'$ appear on its boundary, and each vertex appears exactly once. So there is a unique way to add the edge $bb'$. Note that this puts vertex $c'$ into face $f_3$. Now, vertices $a, a'$ appear exactly once each on the boundary of $f_1$ and $f_2$. This will lead to vertex $c$ to either be in face $f_1$ or $f_2$. In either case, it is impossible to add the edge $cc'$, and hence the given embedding of a subdivision of $\K$ cannot be extended to and embedding of $F_{14}$.
            \begin{figure}[H]
                \centering
                \includegraphics[height=2.5cm]{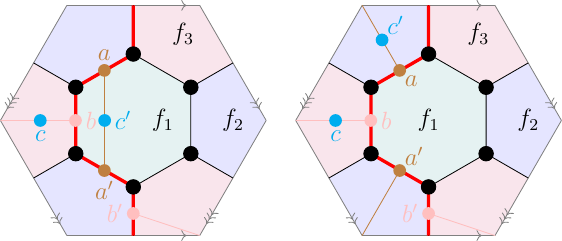}
                \label{fig:F14Case5}
            \end{figure}

    \end{itemize}

    Using the same technique as in the previous sections, we will now show that each of these three embeddings of $F_{14}$ into the torus is equivalent to one of the two embeddings in Figure \ref{fig:TwoF14Embeddings}.

    \noindent Consider the two pentagons, three hexagons, and one octagon below. Identifying vertices with the same label and edges between the same pair of vertices, we obtain the $8$-cycle embedding of $F_{14}$ from Figure \ref{fig:TwoF14Embeddings}.
    \begin{figure}[H]
        \centering
        \includegraphics[height=1.75cm]{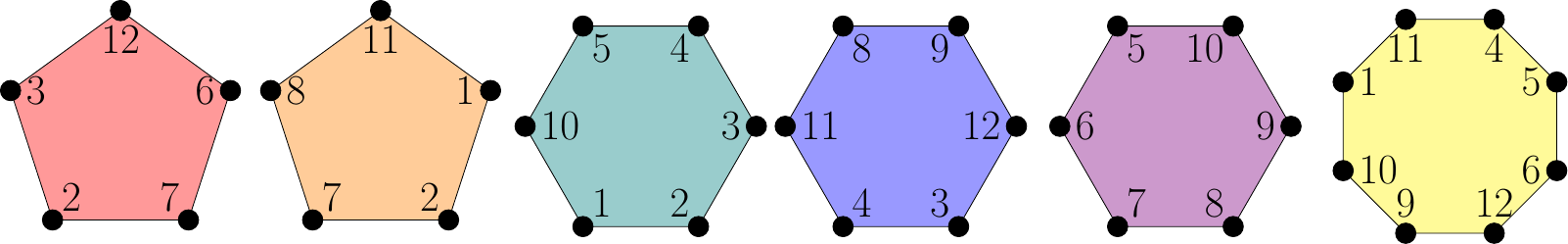} \hspace{.25cm}
        \includegraphics[height=2.5cm]{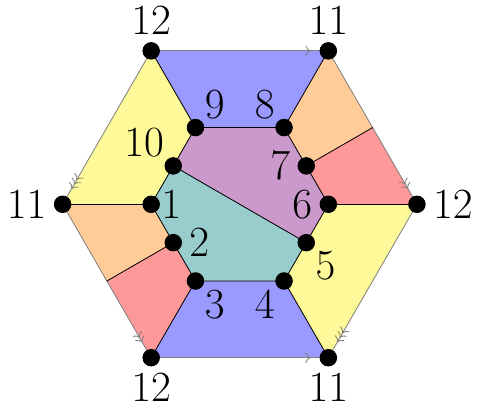}
        \caption{Decomposing the $8$-cycle embedding of $F_{14}$}
        \label{fig:8CycleF14Decomp}
    \end{figure}

    We now observe that we can decompose the embedding of $F_{14}$ resulting from Case 2 into two pentagons, three hexagons, and one octagon. 
        Moreover, we can label the vertices of these embeddings and decompose the torus as in Figure \ref{fig:8CycleF14Decomp}. This implies the existence of a homeomorphism of the torus that realizes an equivalence of the embedding from this case and the known $8$-cycle embedding of $F_{14}$.
        \begin{center}
            \includegraphics[height=3cm]{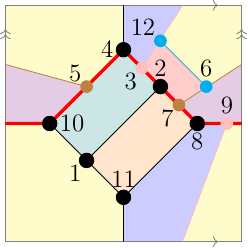}
        \end{center}

    Similarly, consider the four pentagons, one hexagon and once decagon below. Identifying vertices with the same label and edges between the same pair of vertices, we obtain the $10$-cycle embedding of $F_{14}$ from Figure \ref{fig:TwoF13Embeddings}.
    \begin{figure}[H]
        \centering
        \includegraphics[height=1.75cm]{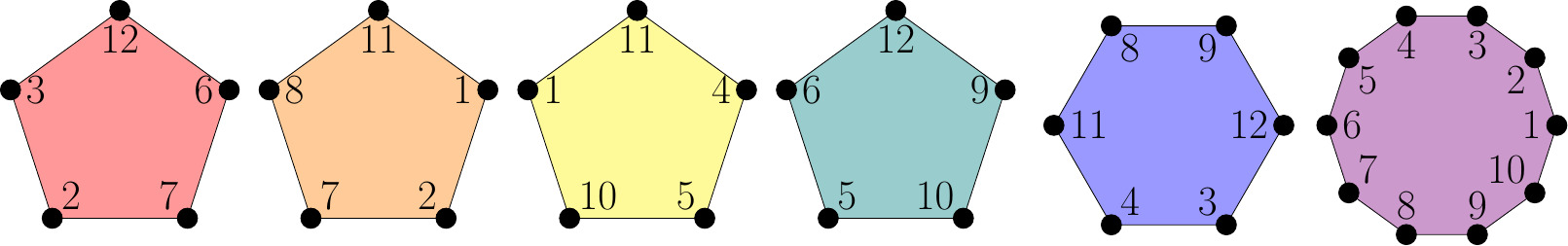} \hspace{.25cm}
        \includegraphics[height=2.5cm]{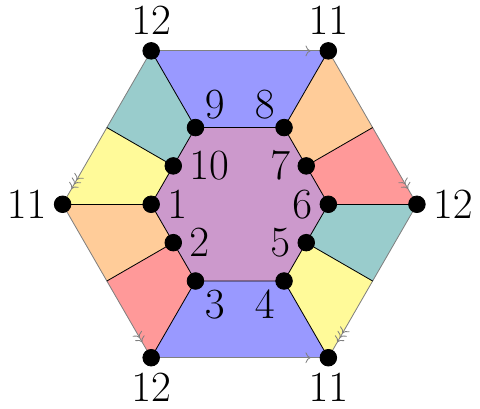}
        \caption{Decomposing the $10$-cycle embedding of $F_{14}$}
        \label{fig:10CycleF14Decomp}
    \end{figure}

    We decompose the two embeddings of $F_{14}$ resulting from Case 4 into four pentagons, one hexagon and once decagon. Moreover, we can label the vertices of these embeddings and decompose the torus as in Figure \ref{fig:7CycleF13Decomp}. Hence, the embeddings from this case are equivalent to the known $10$-cycle embedding of $F_{14}$.
        \begin{center}
            \includegraphics[height=3cm]{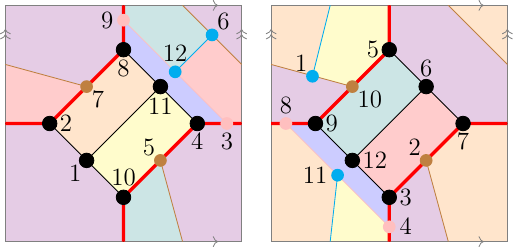}
        \end{center}

    \noindent We can now conclude that there are at most two inequivalent embeddings of $F_{43}$ into the torus, namely the ones depicted in Figure \ref{fig:TwoF14Embeddings}. The remarks preceding the proof of Lemma \ref{lem:F14TorusEmbeddings} imply that these two embeddings are inequivalent, finishing the proof.   
    
\end{proof}

\section{Toroidal embeddings of \texorpdfstring{$G_{1}$}{G\_{1}}}
\label{sec:G1Lemma}

\setcounter{claim}{0}

\begin{lemma}
\label{lem:G1TorusEmbeddings}
    There are exactly two inequivalent unlabelled embeddings of $G_1$ into the torus.
\end{lemma}

\begin{figure}[H]
    \centering
    \includegraphics[height=3cm]{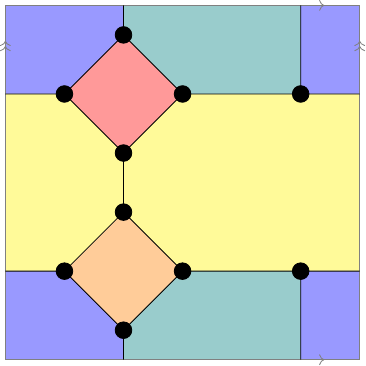} \hspace{.1cm}
    \includegraphics[height=3cm]{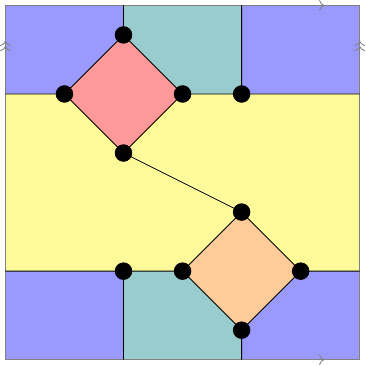}
    \caption{Two embeddings of $G_1$ into the torus}
    \label{fig:TwoG1Embeddings}
\end{figure}

\begin{remark}
    Both embeddings have two facial cycles of length four, two facial cycles of length six, and one facial cycle of length ten.
    To distinguish between the two, we consider the facial cycles of length four. Denote these facial cycles in the embedding on the left by $C_L$ and $C'_L$, and the ones in the embedding on the right by $C_R$ and $C'_R$. Now we observe that there are two distinct edges connecting a vertex in $C_L$ to a vertex in $C'_L$, while there is only one edge connecting a vertex in $C_R$ to one in $C'_R$.  
    As a homeomorphism of the torus would have to map an edge connecting facial $4$-cycles to an edge connecting facial $4$-cycles, the two embeddings cannot be equivalent. We will refer to the first of these two embeddings as the \textit{2-embedding}, and the other as the \textit{1-embedding}. 
\end{remark}

\begin{proof}
    Similar to the other graphs we have analysed in Sections \ref{sec:F11Lemma} -- \ref{sec:F14Lemma}, we fix a subgraph $H$ of $G_1$ homeomorphic to $\K$ as shown below.
        \begin{figure}[H]
        \centering
        \includegraphics[height=2.5cm]{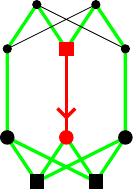}
    \end{figure}

    \noindent The subdivision of $\K$ is shown below, where $\K$ is drawn as the Möbius ladder on three rungs. Note that the directed edge $\vec e$ starts at the unique vertex two of whose incident edges get subdivided twice. Moreover, $\vec e$ does not get subdivided.
    Once we have the pictured subdivision of $\K$, there is a unique way to add the remaining edges to obtain $G_1$ by recalling that $G_1$ is cubic and has girth four.
    \begin{figure}[H]
        \centering
        \includegraphics[height=2cm]{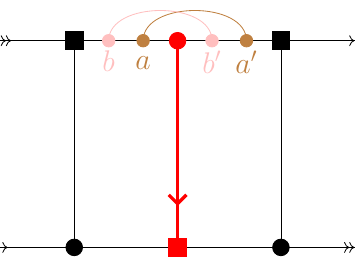}
    \end{figure}

    \noindent By Lemma \ref{lem:ClassDirEdgeK33}, there are six inequivalent directed edged embeddings $(\K, \vec e) \hookrightarrow T^2$. As described above, picking one edge in $\K$ completely determines how to build $G_1$. Hence, we analyse the six cases below and check which can be extended to an embedding of $G_{1}$.
    \begin{figure}[H]
        \centering
        \includegraphics[width=\textwidth]{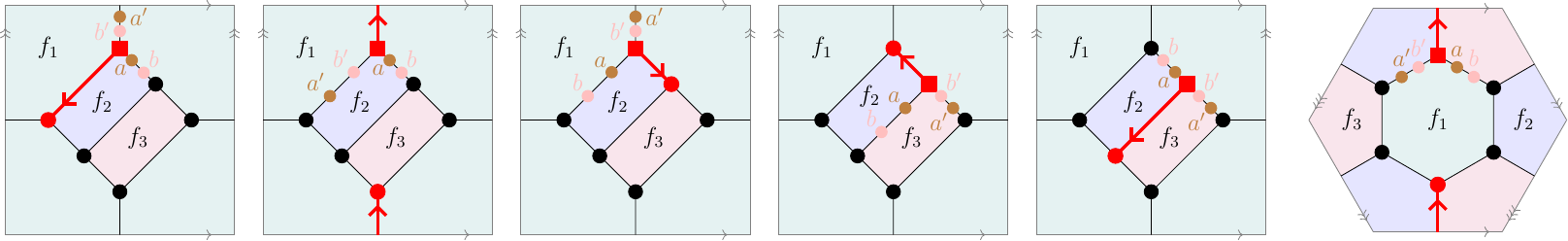}
        \caption{Subdivisions of $\K$ in $G_{1}$}
        \label{fig:K33InG1Cases}
    \end{figure}

     \noindent Note that since all embeddings of $\K$ into the torus are cellular, each face is homeomorphic to a disk. Thus, if two vertices are on the boundary of a common face and neither of them appears more than once, there is a unique way to add an edge between the vertices.

    \begin{itemize}
        \item Case 1: $f_1$ is the unique face such that the vertices $a$ and $a'$ appear on its boundary. $f_1$ is also the unique face with vertices $b$ and $b'$ to appear on its boundary. Moreover, vertices $a'$ and $b'$ appear twice each. Thus, there are two ways to draw the edge $bb'$, and two ways to draw the edge $aa'$. It is quick to see that one of the options of adding the edge $bb'$ makes adding the edge $aa'$ impossible (left picture). Thus, we choose the other option. Then there are two ways of drawing the edge $aa'$ as vertex $a'$ appears twice on the boundary of face $f'_1$ after adding the edge $bb'$. Hence, we get the two embeddings of $G_1$ pictured below.
            \begin{figure}[H]
                \centering
                \includegraphics[height=3cm]{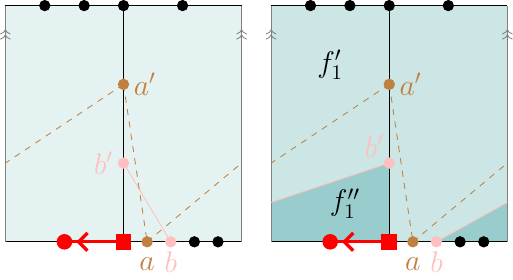} \hspace{.1cm}
                \includegraphics[height=3cm]{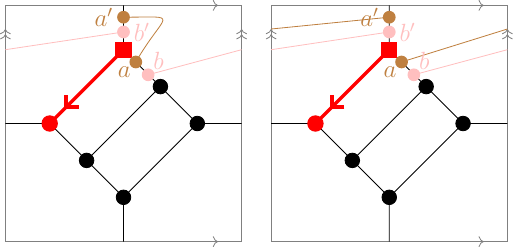}
                \label{fig:G1Case1}
            \end{figure}

        \item Case 2: Vertices $a$ and $a'$ are both on the boundary of face $f_1$ and $f_2$. Moreover, they appear exactly once one each boundary. Hence, there is a unique way to connect them through face $f_1$, and a unique way to connect them through face $f_2$. In both cases, after adding the edge $aa'$ there then is a unique face such that vertices $b$ and $b'$ appear on its boundary, and moreover each appears exactly once. Hence, we get the two embeddings of $G_1$ pictured below.
            \begin{figure}[H]
                \centering
                \includegraphics[height=3cm]{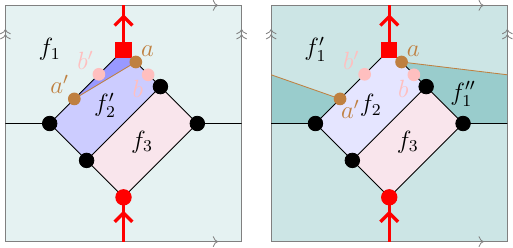} \hspace{.1cm}
                \includegraphics[height=3cm]{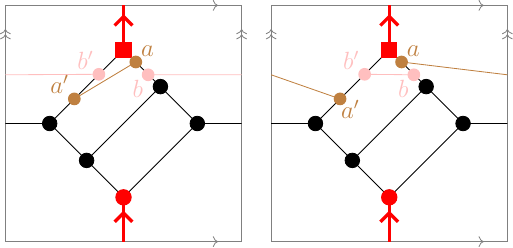}
                \label{fig:G1Case2}
            \end{figure}

        \item Case 3: As in Case 1, $f_1$ is the unique face such that the vertices $a$ and $a'$ appear on its boundary and $b$ and $b'$ appear on its boundary. As $a'$ and $b'$ appear twice, there are two ways to draw the edge $bb'$, and two ways to draw the edge $aa'$. Again, one of the options of adding the edge $bb'$ makes adding the edge $aa'$ impossible (left picture). Thus, we choose the other option, and as in Case 1 we see there are two ways of drawing the edge $aa'$, giving the two embeddings of $G_1$ pictured below.
            \begin{figure}[H]
                \centering
                \includegraphics[height=3cm]{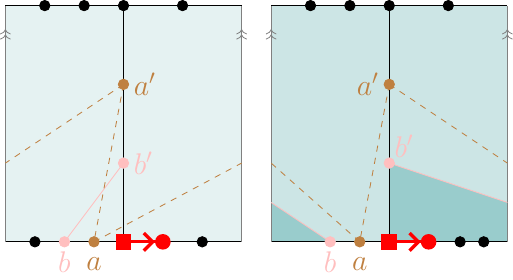} \hspace{.1cm}
                \includegraphics[height=3cm]{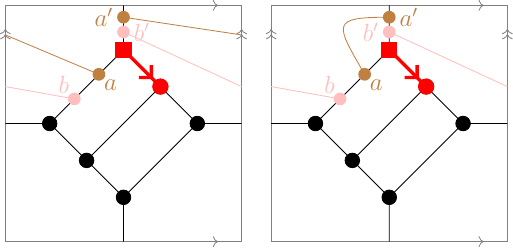}
                \label{fig:G1Case3}
            \end{figure}

        \item Case 4: $f_3$ is the unique face such that the vertices $a$ and $a'$ appear on its boundary, and each appears exactly once. Thus, there is a unique way to draw the edge $aa'$. However, now there is no face with $b$ and $b'$ on its boundary. Hence, the given embedding of a subdivision of $\K$ cannot be extended to an embedding of $G_1$.
            \begin{figure}[H]
                \centering
                \includegraphics[height=3cm]{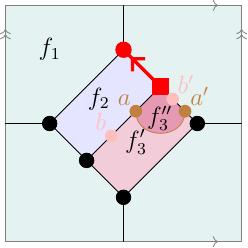}
                \label{fig:G1Case4}
            \end{figure}

        \item Case 5: Similar to Case 4, $f_1$ is the unique face such that the vertices $a$ and $a'$ appear on its boundary, and each appears exactly once. Thus, there is a unique way to draw the edge $aa'$. However, now there is no face with $b$ and $b'$ on its boundary. Hence, the given embedding of a subdivision of $\K$ cannot be extended to an embedding of $G_1$.
            \begin{figure}[H]
                \centering
                \includegraphics[height=3cm]{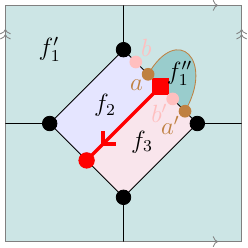}
                \label{fig:G1Case5}
            \end{figure}

        \item Case 6: $f_1$ is the unique face such that the vertices $a, a'$ appear on its boundary. $f_1$ is also the unique face with vertices $b, b'$ to appear on its boundary. Thus, there is a unique way to add each of the edges individually. However, we quickly see that once we have added the edge $aa'$, it is impossible to add the edge $bb'$. Hence, the given embedding of a subdivision of $\K$ cannot be extended to an embedding of $G_1$.
            \begin{figure}[H]
                \centering
                \includegraphics[height=3cm]{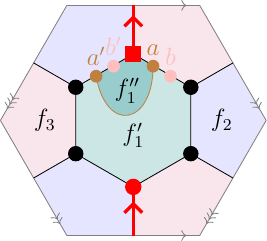} 
                \label{fig:G1Case6}
            \end{figure}
            
    \end{itemize}

    Using the same technique as in the previous sections, we will now show that each of these six embeddings is equivalent to one of the two embeddings in Figure \ref{fig:TwoG1Embeddings}. 

    \noindent Consider the two squares, two hexagons, and one decagon below. Identifying vertices with the same label and edges between the same pair of vertices, we obtain the 2-embedding of $G_1$ from Figure \ref{fig:TwoG1Embeddings}.
    \begin{figure}[H]
        \centering
        \includegraphics[height=2.25cm]{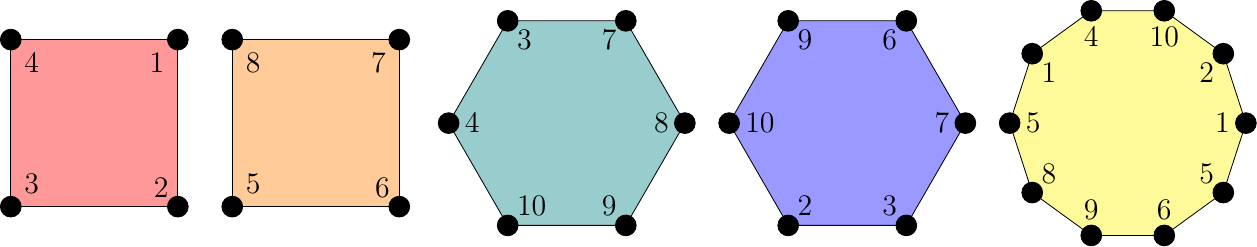} \hspace{.5cm}
        \includegraphics[height=3cm]{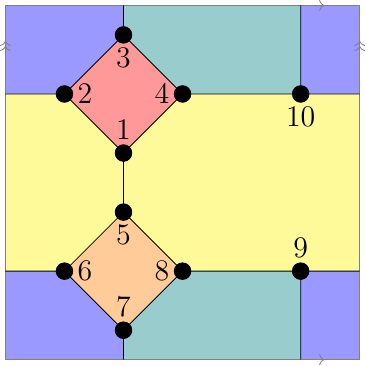}
        \caption{Decomposing the 2-embedding of $G_1$}
        \label{fig:2EmbG1Decomp}
    \end{figure}

    We now observe that we can decompose each of the embeddings of $G_1$ resulting from the above cases into two squares, two hexagons and one decagon.
    Moreover, we can label the vertices of the second embeddings of Cases 1, 2, 3 and decompose the torus as in Figure \ref{fig:1EmbG1Decomp}. This implies the existence of a homeomorphism of the torus that realizes an equivalence of the embedding from these cases and the known 1-embedding of $G_1$.
    \begin{center}
        \includegraphics[height=2.75cm]{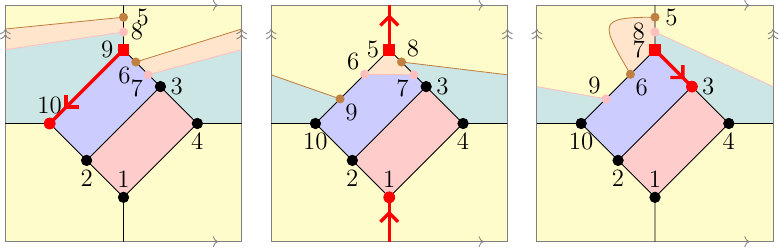}
    \end{center}

    Finally, consider the two squares, two hexagons, and one decagon below. Note that each vertex appears three times, and each edge appears twice. Identifying vertices with the same label and edges between the same pair of vertices, we obtain the 1-embedding of $G_1$ from Figure \ref{fig:TwoG1Embeddings}.
    \begin{figure}[H]
        \centering
        \includegraphics[height=2.25cm]{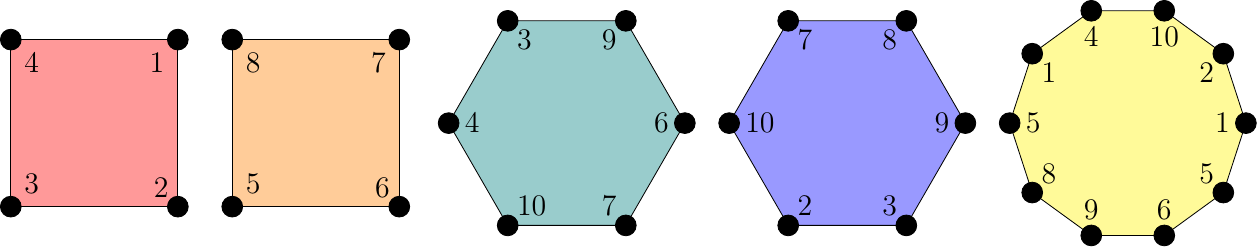} \hspace{.5cm}
        \includegraphics[height=3cm]{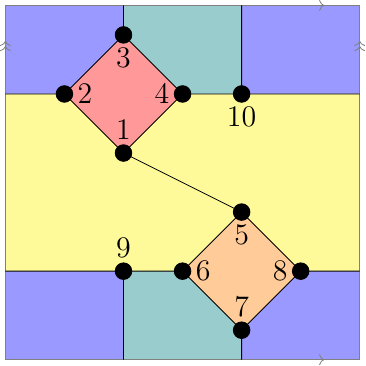}
        \caption{Decomposing the 1-embedding of $G_1$}
        \label{fig:1EmbG1Decomp}
    \end{figure}

    We can label the vertices of the first embeddings of Cases 1, 2, 3 and decompose the torus as in Figure \ref{fig:2EmbG1Decomp}. This implies the existence of a homeomorphism of the torus that realizes an equivalence of the embedding from these cases and the known 2-embedding of $G_1$.
    \begin{center}
        \includegraphics[height=2.75cm]{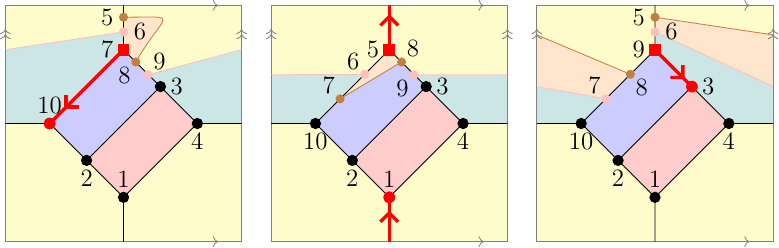}
    \end{center}

    \noindent We can now conclude that there are at most two inequivalent embeddings of $G_1$ into the torus, namely the ones depicted in Figure \ref{fig:TwoG1Embeddings}. The remarks preceding the proof of Lemma \ref{lem:G1TorusEmbeddings} imply that these two embeddings are inequivalent, finishing the proof.  
    
\end{proof}

\appendix
\section*{Appendix. Table of toroidal embeddings of cubic projective plane obstructions}
\label{app:TableRP2ObstructionsInT2}

$E_{42}$ has no embeddings into the torus.

\vspace{.1cm}
\noindent $F_{11}$ has the following two inequivalent unlabelled embeddings into the torus.
\begin{center}
    \includegraphics[height=3cm]{Pics/F11Embeddings/F11Torus1.pdf} \hspace{.5cm}
    \includegraphics[height=3cm]{Pics/F11Embeddings/F11Torus2.pdf}  
\end{center}

\vspace{.1cm}
\noindent $F_{12}$ has the following four inequivalent unlabelled embeddings into the torus.
\begin{center}
    \includegraphics[height=3cm]{Pics/F12Embeddings/F12Torus3_446688.pdf}
    \hspace{.5cm}
    \includegraphics[height=3cm]{Pics/F12Embeddings/F12Torus4_447777.pdf}

    \includegraphics[height=2cm]{Pics/F12Embeddings/F12Torus1_555588.pdf} 
    \hspace{.5cm}
    \includegraphics[height=2cm]{Pics/F12Embeddings/F12Torus2_555588.pdf}
\end{center}

\vspace{.1cm}
\noindent $F_{13}$ has the following two inequivalent unlabelled embeddings into the torus.
\begin{center}
    \includegraphics[height=3cm]{Pics/F13Embeddings/F13Torus1_555669.pdf}
    \hspace{.5cm}
    \includegraphics[height=3cm]{Pics/F13Embeddings/F13Torus2_555777.pdf}  
\end{center}

\vspace{.1cm}
\noindent $F_{14}$ has the following two inequivalent unlabelled embeddings into the torus.
\begin{center}
    \includegraphics[height=3cm]{Pics/F14Embeddings/F14Torus1_556668.pdf}
    \hspace{.5cm}
    \includegraphics[height=3cm]{Pics/F14Embeddings/F14Torus2_5555610.pdf}  
\end{center}

\vspace{.1cm}
\noindent $G_{1}$ has the following two inequivalent unlabelled embeddings into the torus.
\begin{center}
    \includegraphics[height=3cm]{Pics/G1Embeddings/G1TorusStacked.pdf}
    \chead{.5cm}
    \includegraphics[height=3cm]{Pics/G1Embeddings/G1ToursShifted.pdf}  
\end{center}

\bibliographystyle{alpha}
\bibliography{main}

\end{document}